\documentclass[11pt]{article}
\usepackage{amsthm}
\usepackage{amscd}
\usepackage{latexsym}
\usepackage{graphicx}
\usepackage{amsmath}
\usepackage{amsfonts}
\usepackage{float}
\usepackage{amssymb}
\usepackage{enumerate}
\usepackage{multirow}
\usepackage{floatflt}
\usepackage{color}
\usepackage{cite}

\usepackage{enumitem}
\usepackage{algorithmicx}
\usepackage{algpseudocode}
\usepackage[Algorithm,ruled]{algorithm}

\usepackage{algcompatible}
\newtheorem{theorem}{Theorem}
\newtheorem{lemma}{Lemma}
\newtheorem{corollary}{Corollary}
\newtheorem{definition}{Definition}
\newtheorem{proposition}{Proposition}

\theoremstyle{definition}
\newtheorem{example}{Example}
\theoremstyle{remark}
\newtheorem{remark}{Remark}

\newcommand{\R}{\mathbb R}

\newcommand{\bx}{\bar{x}}
\newcommand{\by}{\bar{y}}
\newcommand{\tx}{\tilde{x}}
\newcommand{\ty}{\tilde{y}}
\newcommand{\KX}{K_{\mathbb X}}
\newcommand{\KY}{K_{\mathbb Y}}
\newcommand{\eX}{e_{\mathbb X}}
\newcommand{\eY}{e_{\mathbb Y}}

\newcommand{\wx}{\widetilde{x}}

\newcommand{\ipopt}{\texttt{IPOPT}}

\def \n1{\texttt{np1}}
\def \np2{\texttt{np2}}

\newcommand{\inner}[2]{\langle#1,#2\rangle}
\newcommand{\nn}{\nonumber}
\newcommand{\ind}{i}

\newcommand{\norm}[1]{\textstyle{\Vert} #1 \textstyle{\Vert}}
\newcommand{\pclarke}{\partial^{\tt C}}
\newcommand{\DX}{\Delta_{\mathbb X}}

\begin{document}

\title{Computing  critical angles between two convex cones }

\author{Welington de Oliveira$^\ast$ \and Valentina Sessa\footnote{ Mines Paris, Universit\'e PSL, Centre de Math\'ematiques Appliqu\'ees (CMA),  Sophia Antipolis, France.\\ The research of Valentina Sessa and Welington  de Oliveira benefited from the support of the 
Gaspard-Monge Program for Optimization and Operations Research (PGMO) project
“SOLEM - Scalable Optimization for Learning and Energy Management”.} \and David Sossa\footnote{Universidad de O'Higgins, Instituto de Ciencias de la Ingenier\'ia, Av.\,Libertador Bernardo O'Higgins 611, Rancagua, Chile (e-mail: david.sossa@uoh.cl). Partially supported by  FONDECYT (Chile)  through grant 11220268.}    } 

\date{\today}

\maketitle

\bigskip
\bigskip

\begin{quote}
 {\small \textbf{Abstract}. 

 This paper addresses the numerical computation of critical angles between two convex cones in finite-dimensional Euclidean spaces. We present a novel approach to computing these critical angles by reducing the problem to finding stationary points of a fractional programming problem. To efficiently compute these stationary points, we introduce a partial linearization-like algorithm that offers significant computational advantages. Solving a sequence of strictly convex subproblems with straightforward solutions in several settings  gives the proposed algorithm high computational efficiency while delivering reliable results: our theoretical analysis demonstrates that the proposed algorithm asymptotically computes critical angles. Numerical experiments validate the efficiency of our approach, even when dealing with problems of relatively large dimensions: only a few seconds are necessary to compute critical angles between different types of cones in spaces of dimension 1000. 
\bigskip
\bigskip
\\

{\it Mathematics Subject Classification}:  52A40; 90C26; 90C32; 90C46; 17C99
\\

{\it Keywords}: Critical angle; Convex cone; Fractional Programming; Optimality conditions; Euclidean Jordan algebra. 
}\end{quote}
\bigskip
\bigskip

\section{Introduction}\label{se:intro}
In this work, we are concerned with the problem of computing critical (principal) angles  between two closed convex cones $P\neq \emptyset$ and $Q\neq \emptyset$ in Euclidean spaces. This class of problems, investigated from the theoretical point of view in \cite{SeSo1} and \cite{SeSo2},
finds applications in image classifications and other domains \cite{Sogi}.
In the more straightforward case where $P$ and $Q$ are linear subspaces, the concept of critical angles 
has applications in statistics and numerical linear algebra. 
Such a concept has been abundantly studied in the literature, both from a theoretical and computational
point of view \cite{Miao}. However, to the best of our knowledge, no numerical procedure exists to compute critical angles when $P$ and $Q$ are not linear subspaces but convex cones  in medium or high dimensional spaces. This work contributes to fill the gap by proposing an implementable algorithm for computing critical angles between two convex cones. Furthermore, our framework is general enough to cover several classes of cones, such as polyhedral, ellipsoidal,  Loewnerian cones, and others. Indeed, our methodology applies to the vast family of LISC cones, which are   closed convex cones that can be expressed as linear images of the symmetric cones of certain Euclidean Jordan algebras, plus some mild assumptions (see Definition\,\ref{def:LISC} below).

By definition, the maximal angle between $P$ and $Q$, in a Euclidean space $\mathbb V$ equipped with an inner product $\langle\cdot,\cdot\rangle$,  is the maximum value of $\arccos\,\langle u,v\rangle$ on $(P\cap S_{\mathbb V})\times (Q\cap S_{\mathbb V})$, where $S_{\mathbb V}$ denotes the unit sphere of $\mathbb V$. Critical angles between $P$ and $Q$ are angles formed by vectors $u\in P$ and $v\in Q$ satisfying some complementarity conditions (see Definition\,\ref{def:critical}), which are necessary conditions for $u$ and $v$ to achieve the maximal angle between $P$ and $Q$.

The study of critical angles for the case $P=Q$ was first carried out by Iusem and Seeger \cite{IS1,IS2,IS3,IS4}. Later, Seeger and Sossa expanded this study to two different cones \cite{SeSo1,SeSo2}. See also the recent contributions of Orlitzky\,\cite{Or}  and Bauschke et al.\,\cite{Bau}.
Concerning the computation of critical angles, some methods were given in \cite{SeSo2} for specially structured cones such as revolution cones, ellipsoidal, and cones of matrices. These methods are either applicable to specific instances of cones or limited to problems on spaces of small dimensions. In \cite{SeSo1}, an exhaustive method was developed for computing all critical angles between polyhedral cones. Basically, one must select subsets of generators of the cones, and solve a generalized eigenvalue problem for each selection. This has a combinatorial nature and its numerical implementation is only possible for polyhedral cones with few generators.

 In this work, we show that computing critical angles between two convex cones reduces to the problem of finding stationary points of a fractional programming (FP) problem, that is, the optimization of the ratio of two continuous functions over a nonempty, closed convex set~\cite{FrSc} (see \eqref{maxfrac} for the precise definition).
When specialized to the family of LISC cones,  the feasible set of our FP is the Cartesian product of spectraplexes (sets analogous to the standard simplex) of the involved Euclidean Jordan algebras. Thus, from the numerical point of view, our FP is easier to handle than the original maximal angle problem formulation, whose feasible set has spherical conditions. 
Notably, most of the theoretical and algorithmic works in fractional programming are applied to the convex case, i.e., convex numerator and concave denominator. Dinkelbach's method \cite{D67,S76} is the classical approach used to solve convex fractional problems by tackling a parametric reformulation of the original FP problem.
Several iterative schemes based on Dinkelbach's idea are proposed for special cases of convex FP \cite{Ben,BoCs,DkPo,Scha,StMi}. In the recent paper~\cite{BDLi}, an extrapolated proximal subgradient algorithm is applied to nonconvex and nonsmooth FPs. A disadvantage of that algorithm is that a nonconvex subproblem must be globally solved at every iteration, restricting thus the approach to a few particular FP cases. In theory, that algorithm could be applied to our FP formulation, but the structure of our problem is not (computationally) favorable to that approach. 

This study addresses the challenges posed by our FP, a nonconvex FP problem that does not conform to the requirements of existing practical methodologies. Inspired by the recent paper~\cite{FJOS}, we introduce a Sequential Regularized Partial Linearization algorithm to compute a stationary point of our FP formulation, thus a critical angle. Our algorithm,  that features the main contribution of this work, solves a sequence of two independent convex subproblems, which are projections onto well-structured sets. The essential advantage of our approach lies in the fact that solving these subproblems requires minimal computational efforts as they have straightforward solutions in several settings of practical interest. 
To demonstrate the effectiveness of our approach, we provide numerical results that focus on computing critical angles between LISC cones. These angles are commonly examined in the relevant literature, but a thorough numerical treatment of these problems has not been previously investigated.

This work is organized as follows. Section~\ref{sec:def} presents the main definitions and preliminary results. The concept of LISC cones and relevant properties are presented in Section~\ref{sec:LISC}. Section~\ref{sec:criang} clarifies the connection between critical angles and stationary points of the FP formulation. The proposed algorithm is presented in Section~\ref{sec:alg}
along with a converge analysis. Numerical experiments are reported in Section~\ref{sec:comp}. Finally, some conclusions are given in Section~\ref{sec:conc}.

\section{Main definitions and prerequisites}\label{sec:def}
Let $\mathbb V$, $\mathbb X$ and $\mathbb Y$ be Euclidean spaces. We denote by $\mathcal C(\mathbb V)$ the set of all nontrivial closed convex cones in $\mathbb V$, that is,
\begin{equation}\label{eq:CX}
\mathcal C(\mathbb V)=\{K \mbox{ closed convex cone in }\mathbb V:\,K\neq \{0\},\,K\neq \mathbb V\}.
\end{equation}
The dual cone of $K\in\mathcal C(\mathbb V)$ is denoted by $K^\ast$:
\[
K^\ast:=\{\zeta \in \mathbb V :\, \inner{\zeta}{u}\geq 0,\; \forall\; u \in K\}.
\]
Throughout this work, $\langle\cdot,\cdot\rangle$ refers to the inner product of each Euclidean space that we are considering. The corresponding inner product to its space will be clear from the context. In particular, whenever we consider the Euclidean space $\mathbb R^n$, we assume that its inner product is $\langle x,y\rangle=x^\top y$, for all $x,y\in\mathbb R^n$. The norm $\Vert\cdot\Vert$ refers to the Euclidean norm induced by $\langle\cdot,\cdot\rangle$.

Given two cones $P,Q\in\mathcal C(\mathbb V)$, by definition, the maximal angle between $P$ and $Q$ is given by
\begin{equation}\label{prob:max}
    \Theta(P,Q):=
\left\{
\begin{array}{ll}
\max&\displaystyle \arccos\,\langle u,v\rangle\\
\mbox{s.t.}&u\in P,\,\Vert u\Vert=1\\
&v\in Q,\,\Vert v\Vert=1.
\end{array}\right.
\end{equation}
 As in \emph{data analysis}, the cosinus similarity of two vectors $u,v$ is defined as the cosinus of the angle $\theta$ between them, i.e., $\cos(\theta)=\frac{\inner{u}{v}}{\norm{u}\norm{v}}$. Recall that two proportional vectors have a cosinus similarity of 1, two orthogonal vectors have a similarity of 0, and two opposite vectors have a similarity of -1. Hence,  the cosinus of the maximal angle between $P$ and $Q$ can be computed by solving the following (nonconvex) optimization problem:
\begin{equation}\label{maximal}
\left\{
\begin{array}{ll}
\min&\displaystyle\langle u,v\rangle\\
\mbox{s.t.}&u\in P,\,\Vert u\Vert=1,\\
&v\in Q,\,\Vert v\Vert=1.
\end{array}\right.
\end{equation}
The following definition is crucial for the remainder of this work.
\begin{definition}\label{def:critical}
\textup{(Maximal and critical angles between two cones).}
\begin{enumerate}
\item A pair $(u,v)\in \mathbb{V}\times  \mathbb{V}$ solving problem (\ref{maximal}) is called an \emph{antipodal pair} of $(P,Q)$.  The angle between an antipodal pair is called the maximal angle between the cones $P$ and $Q$.
\item A pair $(u,v)$ is called a \emph{critical pair} of $(P,Q)$ if it satisfies the following conditions:
\begin{equation} \label{crit}
\left\{\begin{array}{lll}
u\in P,\,\Vert u\Vert =1,\\ [1,0mm]
v\in Q,\,\Vert v\Vert=1,\\ [1,0mm]
v-\langle u,v\rangle u\in P^\ast,\\ [1,0mm]
u-\langle u,v\rangle v\in Q^\ast,
\end{array}
\right.
\end{equation}
where  $P^\ast$ and  $Q^\ast$ are the dual cones of $P$ and $Q$, respectively.  
The angle between a critical pair is called a critical angle between the cones $P$ and $Q$. 
\end{enumerate}
\end{definition}
As we will see in Section~\ref{sec:LISC}, in many important cases the cones $P$ 
and $Q$ are images of linear mappings applied to more structured/simpler cones $\KX$ and $\KY$, i.e., $P=G(\KX)$ and $Q=H(\KY)$, with $G$ and $H$ linear applications of appropriated dimensions.
In these cases, it is computationally advantageous to work with the more structured cones $\KX$ and $\KY$. As we will see in Section\,\ref{sec:criang},  problem \eqref{prob:max} can be reformulated as 
\begin{equation}\label{maxfrac}
\cos[\Theta(P,Q)]=
\left\{
\begin{array}{ll}
\min&\displaystyle \Phi(x,y):=\frac{\langle Gx,Hy\rangle}{\left\|Gx\right\| \left\|Hy\right\|}\\
\mbox{s.t.}&x\in K_{\mathbb X},\;\langle e_\mathbb X,x\rangle=1\\
&y\in K_{\mathbb Y},\;\langle e_\mathbb Y,y\rangle=1.
\end{array}\right.
\end{equation}  
In this case, $\KX$ and $\KY$ are the symmetric cones of certain Euclidean Jordan algebras $\mathbb X$ and $\mathbb Y$, and $\eX$ and $\eY$ denote the unit elements of $\mathbb X$ and $\mathbb Y$, respectively.

Observe that the optimization problems  \eqref{maximal} and \eqref{maxfrac} 
share a common structure. Indeed, they fit into the following more general class of problems:
\begin{equation}\label{model}
\left\{
\begin{array}{ll}
\min&\displaystyle\Psi(x,y)\\
\mbox{s.t.}&x\in \KX,\,\phi_1(x)=0\\
&y\in \KY,\,\phi_2(y)=0,
\end{array}
\right.
\end{equation}
where $\Psi:\mathbb{X}\times \mathbb{Y} \to \R$, $\phi_1:\mathbb{X} \to \R$, and $\phi_2:\mathbb{X} \to \R$ are differentiable functions over the feasible set of~\eqref{model}.

\begin{definition}\label{def:stationary}
We say that $(\bar x,\bar y)$ is a \emph{stationary point} of \eqref{model} if there exist $\gamma_1,\gamma_2\in\mathbb R$ such that
\begin{equation}\label{crit2}
\left\{
\begin{array}{l}
\KX\ni \bar x\;\perp\;\left(\nabla_x\Psi(\bar x,\bar y)+\gamma_1\nabla\phi_1(\bar x)\right)\in \KX^\ast,\\
\KY\ni \bar y\;\perp\;\left(\nabla_y\Psi(\bar x,\bar y)+\gamma_2\nabla\phi_2(\bar y)\right)\in \KY^\ast,\\
\phi_1(\bar x)=0,\;\phi_2(\bar y)=0,
\end{array}
\right.
\end{equation}
 where the notation $\perp$ means orthogonality.
\end{definition}

We now show that our concept of stationarity coincides with that of classic stationarity for the nonlinear optimization problem~\eqref{model}.

\begin{theorem}\label{theo:solstat}
Let $\KX\in \mathcal{C}(\mathbb{X})$, $\KY\in \mathcal{C}(\mathbb{Y})$ and 
suppose that $\Psi,\phi_1$ and $\phi_2$  are differentiable over the feasible set of problem~\eqref{model}. Then 
the definition of stationarity given in~\eqref{crit2} coincides with that of stationarity for the nonlinear programming problem~\eqref{model}.
\end{theorem}
\begin{proof}
We can rewrite problem~\eqref{model} with the help of the indicator function of the convex sets $C_1 := \KX \times \mathbb Y$ and $C_2:=\mathbb X \times \KY$ as follows:
\[
\left\{
\begin{array}{ll}
\min&\displaystyle \Psi(x,y) + \ind_{C_1}(x,y) + \ind_{C_2}(x,y)\\
\mbox{s.t.}&\phi_1(x)=0,\; \phi_2(y)=0.
\end{array}\right.
\]
If the above problem has  Lagrange multipliers $\gamma_1,\gamma_2 \in\R$ (e.g., if $(\nabla \phi_1(\bar x),\nabla \phi_2(\bar y))\neq (0,0)$ -- satisfying thus the LICQ constraint qualification), then all of its (local) solutions must satisfy the KKT system (see for instance \cite[\S 10]{RW})
\[
\left\{
\begin{array}{lll}
0 \in \pclarke_{(x,y)} L(\bar x,\bar y,\gamma)\\
\bar x \in \KX,\, \phi_1(\bar x)=0\\
\bar y \in \KY,\, \phi_2(\bar y)=0,
\end{array}
\right.
\]
where the Lagrangian function is given by \[
L(x,y,\gamma):= \Psi(x,y)  + \ind_{C_1}(x,y) + \ind_{C_2}(x,y)+\gamma_1\phi(x)+ \gamma_2\phi_2(y),
\]
and $\pclarke f $ denotes the  Clarke subdifferential of a locally Lipschitz  function $f$. Observe that
\[
 \pclarke_{(x,y)} L(x,y,\gamma)=\begin{pmatrix}
 \nabla_x \Psi(x,y)\\
 \nabla_y \Psi(x,y)
 \end{pmatrix}
 +\pclarke [ \ind_{C_1}(x,y) + \ind_{C_2}(x,y)] + \begin{pmatrix}
 \gamma_1 \nabla \phi_1(x)\\
  \gamma_2 \nabla \phi_2(y)\\
 \end{pmatrix},
\]
and being $\ind_{C_1}(x,y) + \ind_{C_2}(x,y)$ a convex function, its Clarke subdifferential coincides with that of the Convex Analysis. Thus, $\pclarke [ \ind_{C_1}(x,y) + \ind_{C_2}(x,y)] = \partial [ \ind_{C_1}(x,y) + \ind_{C_2}(x,y)]= \partial \ind_{C_2}(x,y) + \partial  \ind_{C_2}(x,y)$ provided
${\tt ri\, dom}( \ind_{C_1})\cap {\tt ri\, dom}( \ind_{C_2}) \neq \emptyset$ (see \cite[Thm 23.8]{Rock}).
Furthermore, ${\tt ri\, dom}( i_{C_1})={\tt ri }\KX \times \mathbb{Y}$ and
${\tt ri\, dom}( i_{C_2})= \mathbb{X} \times {\tt ri }\KY$. As $\KX\neq \{0\}$ and $\KY\neq \{0\}$ (c.f. \eqref{eq:CX}) then  
${\tt ri }\KX\neq \emptyset$ and ${\tt ri }\KY\neq \emptyset$. Recall that $C_1$ is a convex set, and thus
$ \partial \ind_{C_1}(x,y) = N_{C_1}(x,y)=N_{\KX}(x)\times 0$. Analogously,  $ \partial \ind_{C_2}(x,y) = 0 \times N_{\KY}(y)$.
Hence, the above KKT system becomes
\begin{equation}\label{aux0}
\left\{
\begin{array}{lll}
0 \in \begin{pmatrix}
 \nabla_x \Psi(\bar x,\bar y)\\
 \nabla_y \Psi(\bar x,\bar y)
 \end{pmatrix}
 +
 \begin{pmatrix}
  N_{\KX}(\bar x)\\
  N_{\KY}(\bar y)
 \end{pmatrix}
 +
  \begin{pmatrix}
 \gamma_1 \nabla \phi_1(\bar x)\\
  \gamma_2 \nabla \phi_2(\bar y)\\
 \end{pmatrix}
 \\
\bar x \in \KX,\, \phi_1(\bar x)=0\\
\bar y \in \KY,\, \phi_2(\bar y)=0.
\end{array}
\right.
\end{equation}
Then, by using the well-known results in~\cite[Prop.~1.1.3]{FP}, we can see that the above system is nothing but the one given in~\eqref{crit2}.

\end{proof}
It is not surprising that~\eqref{crit} is a particular case of~\eqref{crit2}, as the following corollary shows.
\begin{corollary}
Under the setting of Theorem~\ref{theo:solstat},
    suppose that $\KX=P$, $\KY=Q$, $\Psi(x,y)=\inner{x}{y}$, $\phi_1(x)=\norm{x}-1$, and $\phi_1(y)=\norm{y}-1$, i.e., problem~\eqref{model} coincides with~\eqref{maximal}. Furthermore, suppose that $(\bx,\by)$ is a stationary point of problem~\eqref{model}. Then, the Lagrange multipliers $\gamma_1,\gamma_2 \in \R$ in~\eqref{crit2} are given by
    $
    \gamma_1=\gamma_2=-\inner{\bar x}{\bar y}.
    $
\end{corollary}
\begin{proof}
In this setting, the Lagrange multipliers always exist due to the fact that
 $(\nabla \phi_1(\bar x),\nabla \phi_2(\bar y))=({\bx}/{\norm{\bx}},{\by}/{\norm{\by}})\neq (0,0)$, i.e., the problem satisfies the LICQ constraint qualification.
 Furthermore,  system~\eqref{aux0} becomes
\[
\left\{
\begin{array}{lll}
0 \in \begin{pmatrix}
 \bar y\\
 \bar x
 \end{pmatrix}
 +
 \begin{pmatrix}
 N_P(\bar x)\\
 N_Q(\bar y)
 \end{pmatrix}
 +
  \begin{pmatrix}
 \gamma_1 \bar x/\norm{\bx}\\
 \gamma_2 \bar y/\norm{\by}\\
 \end{pmatrix}
 \\
\bar x \in P,\, \Vert \bar x\Vert=1\\
\bar y \in Q,\, \Vert \bar y\Vert=1.
\end{array}
\right.
\]
Let $\zeta \in N_P(\bar x)$ such that $0=\bar y + \zeta +  \gamma_1 \bar x/\norm{\bx}$. By multiplying this equality by $\bar x$ and recalling that $\Vert \bar x\Vert = 1$, we get $0=\langle \bar x, \bar y\rangle+\langle \bar x, \zeta \rangle+  \gamma_1$. As $\langle \bar x, \zeta \rangle =0$, we have that
 $\gamma_1 =  -\langle \bar x, \bar y\rangle$. Analogously,  $ \gamma_2 =  -\langle \bar x, \bar y\rangle$. 
\end{proof}

It is straightforward to see that, in the setting of the above corollary, the conditions in~\eqref{crit2} boils down to~\eqref{crit}.

\section{LISC cones}\label{sec:LISC}

This section presents the concept of LISC cones which are sets obtained as \emph{Linear Image of Symmetric Cones}. Recall that a \emph{symmetric cone} in a Euclidean space is a cone that is self-dual and homogeneous, see \cite[Section\,I.1]{FK}. The symmetric cones possess a reach structure and they are central objects in many optimization problems such as linear programming, second-order cone programming, and semidefinite programming \cite{VaBo, AlGo,RS}. For those programs, the theory and the numerical resolution are well-developed and nowadays there are many solvers available to deal with them. It turns out that many convex cones arise as the linear image of symmetric ones. As we shall see in the next sections, the computation of critical angles for these cones can be performed by solving a sequence of (convex) symmetric cone programs.

\subsection{Main definitions and some examples}
It is known that every symmetric cone is obtained as the cone of square elements of a certain Euclidean Jordan Algebra (EJA), see \cite[Theorem\,III.3.1]{FK}.   Thus, to present our results we shall use the machinery of EJAs.  We refer the reader to the book of Faraut and Kor\'anyi \cite{FK} for an exposition of the definitions and main results of this algebra.

Let $\mathbb X$ be an EJA of rank $r$ with unit element $e$, and equipped with an inner product $\langle\cdot,\cdot\rangle$ and a Jordan product $\circ$.
 We need to recall some ingredients from EJA concerning the spectral decomposition \cite[Theorem\,III.1.2]{FK}.  An element $c\in\mathbb X$ is called \emph{idempotent} if $c\circ c=c$. An idempotent element $c$ is called \emph{primitive} if it is non-zero and cannot be written as the sum of two non-zero idempotents. A collection of primitive idempotent elements $\{c_1,\ldots,c_r\}$ is called \emph{Jordan frame} if $c_i\circ c_j=0$ for all $i\neq j$ and $\sum_{i=1}^r c_i=e$. Any element $x$ in $\mathbb X$  admits a spectral decomposition. That is, there exist scalars $\lambda_1,\ldots,\lambda_r$ (called the eigenvalues of $x$) and a Jordan frame $\{c_1,\ldots,c_r\}$ such that $x=\sum_{i=1}^r\lambda_i c_i$. The trace of $x=\sum_{i=1}^r\lambda_i c_i$ is defined as ${\rm tr}(x)=\lambda_1+\cdots+\lambda_r$.
 
 \begin{definition}
     [Symmetric cone in $\mathbb X$]
     We denote by $K_\mathbb X$ the symmetric cone in $\mathbb X$, which corresponds to the cone of square elements of $\mathbb X$, that is,  $K_{\mathbb X}=\{x\circ x:x\in\mathbb X\}$.
 \end{definition} 
 Observe that $x=\sum_{i=1}^r\lambda_i c_i\in K_{\mathbb X}$ if and only if $\lambda_i\geq 0$ for all $i=1,\ldots,r$.
 Below, we provide some examples of EJAs and corresponding symmetric cones.
\begin{example}[EJAs and corresponding symmetric cones] \label{ex:EJA}\mbox{ }
\begin{enumerate}[label=(\alph*)]
    \item The space $\mathbb R^n$ is an EJA of rank $n$ with unit element ${\bf 1}_n=(1,\ldots,1)^\top$,  $\langle x,y\rangle=x^\top y$, and $x\circ y=(x_1y_1,\ldots,x_ny_n)^\top$. Its associated symmetric cone is the nonnegative orthant $\mathbb R^n_+$. 
    \item Alternatively, $\mathbb R^n=\mathbb R^{n-1}\times\mathbb R$ is an EJA (Jordan spin algebra) of rank $2$ with unit element given by $e=(0,\ldots,0,1)^\top$. For $x=(\xi,t),\, y=(\eta,s)\in\mathbb R^{n-1}\times\mathbb R$, the inner product is $\langle x,y\rangle=x^\top y$ and  $x\circ y=(s\xi+t\eta,x^\top y)$. Its corresponding symmetric cone is the \emph{Lorentz cone}, which is given by $\mathcal L^n_+:=\{(\xi,t)\in\mathbb R^n:\Vert\xi\Vert\leq t\}$.
    \item The space $\mathcal S^n$ of the symmetric matrices of order $n$ is an EJA of rank $n$ with unit element $I$ (the identity matrix of order $n$), equipped with the Frobenius (or trace) inner product $\langle X,Y\rangle={\rm Tr}(XY)$ and the Jordan product $X\circ Y=(XY+YX)/2$. Its associated symmetric cone is the \emph{positive semidefinite (SDP) cone} which is defined as follows, $\mathcal P_n:=\{X\in\mathcal S^n\,:\,u^\top Xu\geq 0,\forall u\in\mathbb R^n\}.$
\end{enumerate}
\end{example}

A cone $P$ in a Euclidean space $\mathbb V$ is a LISC cone if it can be written as the linear image of some symmetric cone satisfying some mild assumptions. More precisely:
\begin{definition}\label{def:LISC}
Let $P$ be a closed convex cone in a Euclidean space $\mathbb V$. We say that $P$ is a \emph{LISC cone} if there exist a Euclidean Jordan algebra $\mathbb X$ and a linear map $G:\mathbb X\to\mathbb V$ such that $G(K_{\mathbb X})=P$,  satisfying the following assumptions:
\begin{align}
    &\mbox{$G(c)\neq 0$ for every primitive idempotent element $c\in\mathbb X$.}\tag*{(A1)}\label{assump1}\\
    &\mbox{$P$ is pointed (i.e. $P\cap(-P)=\{0\}$).}\tag*{(A2)}\label{assump2}
\end{align}

\end{definition}

Observe that in the above definition, we are assuming that $P=G(\KX)$ is closed in advance. This is because in general, the linear image of $\KX$ may fail to be closed (cf. \cite{Pataki}).

There are vast classes of convex cones that are LISC cones.  We now list some of them.
\begin{example}
[LISC cones]\mbox{ }
\begin{enumerate}[label=(\alph*)]
\item \emph{Polyhedral cones} \cite{Greer}. A cone $P$ in $\mathbb R^n$ is polyhedral if there exist vectors $g_1,\ldots,g_p\in\mathbb R^n$ (generators of $P$) such that 
$$P={\rm cone}\{g_1,\ldots,g_p\}:=\{\alpha_1 g_1+\cdots+\alpha_p g_p:\alpha_i\geq 0,\forall i=1,\ldots,p\}.$$
Then, by taking $G:=[g_1\,\cdots\,g_p]$ as the $n\times p$ matrix formed by the generators of $P$, we have that $P=G(\mathbb R^p_+)$. It is usual to assume that the generators of $P$ are positive linear independent (i.e., if $\sum_{i=1}^p\alpha_i g_i=0$ and $\alpha_i\geq 0$ for all $i$, then $\alpha_i=0$ for all $i$). This property implies \ref{assump1} and \ref{assump2}. Indeed, the primitive idempotent elements of $\mathbb R^p$ are $e_1,\ldots,e_p$ which are the elements of the canonical basis of $\mathbb R^p$. Observe that $G(e_i)=g_i\neq 0$ for every $i$. Then, \ref{assump1} holds. On the other hand, suppose that \ref{assump2} is not satisfied. That is, suppose that $P$ is not pointed. Then, there exists $u\neq 0$ such that $u=\sum_{i=1}^p\alpha_i g_i$ and   $-u=\sum_{i=1}^p\beta_i g_i$ for some $\alpha_i,\beta_i\geq 0$ for all $i$. Then, $\sum_{i=1}^p(\alpha_i+\beta_i) g_i=0$. Since we assume that the generators of $P$ are positive linear independents, we conclude that $\alpha_i=\beta_i=0$ for all $i$. It means $u=0$ which is a contradiction. Hence, $P$ is a LISC cone with the positive orthant $\mathbb R^p_+$ as its associated symmetric cone. 
\item \emph{Ellipsoidal cones} \cite{IS2,SeTo,SeSo2}. A cone $P$ in $\R^n$ is ellipsoidal if there exists a symmetric positive definite matrix $A$ of order $n-1$ such that
\begin{equation*}
    P= \{(\xi,t) \in \R^{n-1}\times\R: \sqrt{\langle \xi, A \xi} \rangle \leq t \}.
\end{equation*}
By taking the linear map $G:\mathbb R^n\to\mathbb R^n$ given by $G(\xi,t)=(A^{-1/2}\xi,t)$, we have that $P=G(\mathcal L^n_+)$. Since $G$ is invertible and $\mathcal L^n_+$ is pointed, we have that \ref{assump1} and \ref{assump2} are satisfied. Thus, $P$ is a LISC cone with the Lorentz cone $\mathcal L^n_+$ as its associated symmetric cone.

\item \emph{Loewnerian cones} \cite{SeSo0}. By definition, a cone $P$ in $\mathcal S^n$ is Loewnerian if there exists an invertible linear map $G:\mathcal S^n\to\mathcal S^n$ such that $P=G(\mathcal P_n)$. Assumptions \ref{assump1} and \ref{assump2} hold because $G$ is invertible and $\mathcal P_n$ pointed. Then, $P$ is a LISC cone with the SDP cone $\mathcal P_n$ as its associated symmetric cone.
\end{enumerate}
\end{example}
 As we will shortly see, our technique to find critical angles between a pair of LISC cones consists of reformulating problem \eqref{prob:max} as a minimization problem over a convex set. More precisely, for a pair of LISC cones $P=G(\KX)$ and $Q=H(\KY)$, the (bilinear) minimization problem \eqref{maximal} with nonconvex feasible set 
 $$S_{P,Q}:=\{(u,v)\in P\times Q:\Vert u\Vert=1,\,\Vert v\Vert=1\}$$
 will be reformulated as a (fractional) minimization problem \eqref{maxfrac} with convex set
 $$\Delta_{\mathbb X,\mathbb Y}:=\{(x,y)\in\KX\times\KY:\langle \eX,x\rangle=1,\,\langle \eY,y\rangle=1\},$$
 where $\eX$ and $\eY$ are the unit elements of the EJAs $\mathbb X$ and $\mathbb Y$, respectively.
Observe that $S_{P,Q}=S_P\times S_Q$, where 
$$S_P:=\{u\in P:\Vert u\Vert=1\}\quad\mbox{and}\quad S_Q:=\{v\in Q:\Vert v\Vert=1\}.$$
Analogously, $\Delta_{\mathbb X,\mathbb Y}=\Delta_{\mathbb X}\times \Delta_{\mathbb Y}$, where 
$$\Delta_{\mathbb X}:=\{x\in\KX:\langle \eX,x\rangle=1\}\quad \mbox{and}\quad \Delta_{\mathbb Y}:=\{y\in\KY:\langle \eY,y\rangle=1\}.$$
Here, $\Delta_{\mathbb X}$ and $\Delta_{\mathbb Y}$ are called the \emph{spectraplexes} of $\mathbb X$ and $\mathbb Y$, respectively. 
A key fact to formulate the maximal angle problem as the minimization problem \eqref{maximal} is that each nonzero $u\in P$ can be written as a positive multiple of an element of $S_P$; indeed, $u=\Vert u\Vert\frac{u}{\Vert u\Vert}$ and $\frac{u}{\Vert u\Vert}\in S_P$ (an analogous property 
 is for a nonzero $v\in Q$). In the same way, we must ensure that each nonzero $x\in \KX$ can be written as a positive multiple of an element of $\Delta_\mathbb X$ (an analogous property for a nonzero $y\in\KY$). This is proved in the following lemma.

\begin{lemma}\label{positive}
Let $\mathbb X$ be an EJA of rank $r$ with unit element $\eX$. Let $\KX$ be the symmetric cone of $\mathbb X$.   Then, the following statements are satisfied:
\begin{enumerate}
    \item $\langle e_{\mathbb X}, x\rangle > 0$ for all $x\in K_\mathbb X\setminus\{0\}$.
    \item $\{\alpha x:x\in \Delta_\mathbb X,\,\alpha>0\}=K_\mathbb X\setminus\{0\}$.
\end{enumerate}
\end{lemma}
\begin{proof}
Let $x=\sum_{i=1}^r\lambda_i c_i$ be a nonzero element in $ K_\mathbb X$. Then, $\lambda_i\geq 0$ for all $i=1,\ldots,r$, and not all of them are zero. Thus,
$$\langle e_{\mathbb X}, x\rangle=\left\langle \sum_{i=1}^r c_i,\sum_{i=1}^r\lambda_i c_i\right\rangle=\sum_{i=1}^n\lambda_i\Vert c_i\Vert^2\geq 0.$$
Then, if $\langle e_{\mathbb X},x\rangle=0$ we have that $\lambda_i\Vert c_i\Vert^2=0$ for every $i=1,\ldots,r$. Since any primitive element is nonzero, we deduce that $\lambda_i=0$ for all $i=1,\ldots,r$ which contradicts the fact that $x$ is nonzero. We conclude $\langle e_{\mathbb X},x\rangle>0$.

Let us prove the second part. Since $\Delta_\mathbb X$ is contained in $K_\mathbb X\setminus\{0\}$ and because of the cone structure of $K_\mathbb X\setminus\{0\}$, we have that $\{\alpha x:x\in \Delta_\mathbb X,\,\alpha>0\}\subseteq K_\mathbb X\setminus\{0\}$. Let $u\in K_\mathbb X\setminus\{0\}$. From the previous part, we have that $\langle e_{\mathbb X}, u\rangle > 0$. Then, $u=\alpha x$ with $\alpha:= \langle e_{\mathbb X}, u\rangle>0$ and $x:=u/\langle e_{\mathbb X}, u\rangle\in \Delta_\mathbb X$. The proof is complete.
\end{proof}

 Now, to pass from  problem~\eqref{prob:max} (which is equivalent to~\eqref{maximal})  to the fractional one \eqref{maxfrac}, we must ensure that $Gx\neq 0$ for every $x\in K_\mathbb X\setminus\{0\}$ because we shall need to normalize $Gx$ (an analogous observation is for $Hy$). The next lemma shows that assumptions \ref{assump1} and \ref{assump2} help us to achieve that property.
But before, observe that a general LISC cone can also be described by a sort of conic combination of a  (possibly uncountable) set of vectors, similar to the polyhedral case. Indeed, let $\mathcal O_\mathbb X$ denote the set of (ordered) Jordan frames of $\mathbb X$. That is, $(c_1,\ldots,c_r)\in\mathcal O_\mathbb X$ if and only if $\{c_1,\ldots,c_r\}$ is a Jordan frame of $\mathbb X$. From the spectral decomposition property, it is not difficult to see that $K_\mathbb X$ can be written as 
$$K_\mathbb X=\left\{\sum_{i=1}^r\alpha_ic_i:\alpha_1,\ldots,\alpha_r\geq 0, (c_1,\ldots,c_r)\in\mathcal O_\mathbb X \right\}=\bigcup_{(c_1,\ldots,c_r)\in\mathcal O_\mathbb X}{\rm cone}\{c_1,\ldots,c_r\}.$$
Then, the LISC cone $P=G(\KX)$ can be written as
\begin{equation}\label{polgen}
P=\bigcup_{(c_1,\ldots,c_r)\in\mathcal O_\mathbb X}{\rm cone}\{G(c_1),\ldots,G(c_r)\}.
\end{equation}
Observe that in the polyhedral case, the representation \eqref{polgen} coincides with the conic combination of its generators since the only Jordan frame of $\mathbb R^r$ is its canonical basis. 
 \begin{lemma}\label{lem:lisc}
 Let $P=G(\KX)$ be a LISC cone in $\mathbb V$ with $\mathbb X$ being an EJA of rank $r$. Then, $Gx\neq 0$ for every $x\in K_\mathbb X\setminus\{0\}$.
\end{lemma}
\begin{proof}
   Reasoning by contradiction, suppose that there is $x=\sum_{i=1}^r\lambda_i c_i$ in $ K_\mathbb X\setminus\{0\}$ such that $G(x)=0$. Then, $\sum_{i=1}^r \lambda_iG(c_i)=0$. Since $x$ is nonzero, we may assume that $\lambda_1\neq 0$. Thus, 
   $$-G(c_1)=\sum_{i=2}^r\left(\frac{\lambda_i}{\lambda_1}\right)G(c_i).$$
   Because of $x\in\KX$, we have that $\lambda_i\geq 0$ for every $i$. Then, from representation \eqref{polgen} we deduce that $-G(c_1)\in P$. On the other hand, it is clear, also for the representation \eqref{polgen}, that $G(c_1)\in P$. Because of the assumption \ref{assump1}, we have that $G(c_1)\neq0$. We have proved that $P\cap(-P)$ has a nonzero element which contradicts assumption \ref{assump2}.
\end{proof}

As a last ingredient of EJA, we now explain how to compute the orthogonal projection of $b\in\mathbb X$ onto the  spectraplex $\DX=\{x\in\KX:\langle \eX,x\rangle=1\}$. This result will be used in our algorithmic approach as we shall see in the next sections. 

\subsection{Projection onto spectraplexes}\label{sec:proj}
An EJA is said to be \emph{scalarizable} if there exists $\kappa>0$ such that $\langle x,y\rangle=\kappa{\rm tr}(x\circ y)$ for every $x,y\in\mathbb X$. The number $\kappa$ is called the scaling factor of the EJA. The EJAs in Example\,\ref{ex:EJA} are scalarizable: the scaling factors of $\mathbb R^n$, $\mathbb R^{n-1}\times\mathbb R$ and $\mathcal S^n$ are $1$, $1/2$ and $1$, respectively. From now on, we assume that all the EJAs considered in this work are scalarizable. For $x\in\mathbb X$, $\lambda(x)$ denotes the vector in $\mathbb R^r$ whose components are the eigenvalues of $x$ arranged in nonincreasing order, i.e., $\lambda_1(x)\geq\lambda_2(x)\geq\cdots\geq\lambda_r(x)$ are the eigenvalues of $x$. A set $\Omega\subseteq\mathbb X$ is called \emph{spectral set} if there exists a permutation invariant set $\mathcal Q\subseteq\mathbb R^r$ such that $\Omega=\lambda^{-1}(\mathcal Q)$. Observe that $\DX$ is a spectral set whose associated permutation invariant set is 
\[\Delta_r^\kappa := \left\{\lambda\in\R^r: \lambda~\geq~0, \:\langle{\bf 1}_r, \lambda\rangle~=~\frac{1}{\kappa}\right\}.\] Indeed, $x\in\DX$ if and only if $\lambda_i(x)\geq 0$ for all $i$, and 
$$1=\langle\eX,x\rangle=\kappa {\rm tr}(\eX\circ x)=\kappa{\rm tr}(x)=\kappa(\lambda_1(x)+\cdots+\lambda_r(x)).$$
It means $\lambda(x)\geq 0$ and $\langle {\bf 1}_r,\lambda(x)\rangle =\frac{1}{\kappa}$. Thus, $\DX=\lambda^{-1}(\Delta_r^\kappa)$. Observe that when $\kappa=1$, which is the case for the algebras $\mathbb R^n$ and $\mathcal S^n$, $\Delta_r^\kappa$ coincides with the standard simplex 
$$\Delta_r := \{x\in\R^r: x~\geq~0, \:\langle{\bf 1}_r, x\rangle~=~1\}.$$

Incidentally, we observe that $\DX$ is compact. Indeed, it is clear that $\Delta_r^\kappa$ is compact in $\mathbb R^r$. Then, the compactness of $\Delta_r^\kappa$ is transferred to its associated spectral set $\DX$ \cite[Theorem\,27]{Baes}. This result will be used later, so we state it as a proposition.
\begin{proposition}\label{compact}
    The spectraplex $\DX$ is compact in $\mathbb X$.
\end{proposition}

The next proposition shows that the computation of the orthogonal projection onto $\DX$ is reduced to the computation of the orthogonal projection onto $\Delta_r^\kappa$ which is easier to handle.
\begin{proposition}\label{prop:proj}
Let $\mathbb X$ be a scalarizable EJA with rank $r$ and scaling factor $\kappa$. Let $b\in\mathbb X$ with spectral decomposition $b=\sum_{i=1}^r\lambda_i(b)c_i$. Then, $\bx=\sum_{i=1}^r\bar\xi_i c_i$ is the solution of
\begin{equation}\label{projection}
\begin{array}{ll}
\min&\displaystyle \frac{1}{2}\Vert x-b\Vert^2\\
\mbox{s.t.}&x\in \DX,
\end{array}
\end{equation} 
where $\bar\xi:=(\bar\xi_1,\ldots,\bar\xi_r)^\top$ is the solution of
\begin{equation}\label{projection2}
\begin{array}{ll}
\min&\displaystyle \frac{1}{2}\Vert \xi-\lambda(b)\Vert^2\\
\mbox{s.t.}&\xi\in \Delta_r^\kappa.
\end{array}
\end{equation}
\end{proposition}
\begin{proof}
The proof is analogous to the proof of \cite[Proposition\,10]{RSS} where the main tool  is the commutation principle in EJA developed in \cite{RSS}.   
\end{proof}
\begin{example}[Projections]\label{ex:pyeccion}\mbox{ }
    \begin{enumerate}[label=(\alph*)]
        \item \label{exa} For the EJA $\mathbb R^n$,  problems \eqref{projection} and \eqref{projection2} are the same. Then, the projection of $b\in\mathbb R^n$ onto the spectraplex of $\mathbb R^n$ is nothing but the projection of $b$ onto the standard simplex $\Delta_n$. That is, the solution to the problem:
        \begin{equation}\label{ex1}
\begin{array}{ll}
\min&\displaystyle \frac{1}{2}\Vert x-b\Vert^2\\
\mbox{s.t.}&x\geq 0,\;\langle{\bf 1}_n,x\rangle=1.
\end{array}
\end{equation} 
The solution $\bar x$ of \eqref{ex1} can be characterized as  
$\bar x_i=\max\{0,b_i+\mu\}\quad\mbox{for all $i=1,\ldots,n$},$
where $\mu\in\mathbb R$ is a solution of $\sum_{i=1}^n\max\{0,b_i+\mu\}=1.$ 
As shown in \cite{Cond}, the point $\bar x$ can be computed efficiently and exactly by specialized algorithms.
\item \label{exb} For the Jordan spin algebra $\mathbb R^{n-1}\times \mathbb R$, \eqref{projection2} is formulated in $\mathbb R^2$ (because the rank of the algebra is $2$) and the solution can be calculated as follows: for $b=(\tilde b,b_0)\in \mathbb R^{n-1}\times \mathbb R$, the solution of \eqref{projection2} is
\begin{equation}\label{solspin}
    (\bar{\xi},\bar{t})=\left(\frac{\tilde b}{\max\{1,\Vert \tilde  b\Vert\}},1\right).
\end{equation}
As a remark for this case, if we work directly \eqref{projection}, we can deduce that $(\bar{\xi},\bar{t})$ is the solution of \eqref{projection} if and only if $\bar{t}=1$ and $\bar{\xi}$ is the orthogonal projection of $\tilde b$ onto the unit ball of $\mathbb R^{n-1}$. That is, $\bar\xi $ is the solution of 
\begin{equation}\nonumber
\begin{array}{ll}
\min &  \frac{1}{2} \| \xi - \tilde b \|^2  \\
s.t.& \|\xi \|^2 \leq 1.
\end{array}
\end{equation}
Thus, $(\bar{\xi},\bar{t})$ coincides with \eqref{solspin}.
 \item \label{exc} For the EJA $\mathcal S^n$, we have that for $B\in\mathcal S^n$,  problem \eqref{projection} becomes
 \begin{equation}\label{projsym}
\begin{array}{ll}
\min&\displaystyle \frac{1}{2}\Vert X-B\Vert^2\\
\mbox{s.t.}&X\in\mathcal P_n,\\
&{\rm Tr}(X)=1.
\end{array}
\end{equation} 
From Proposition\,\ref{prop:proj}, if $B$ has the spectral decomposition $B=\sum_{i=1}^n\lambda_i(B)u_iu_i^T$, where $\{u_1,\ldots,u_n\}$ is an orthonormal basis of $\mathbb R^n$ of eigenvectors of $B$, then $\bar X=\sum_{i=1}^n\bar \xi_iu_iu_i^T$ is the solution of \eqref{projsym}, where $\bar\xi$ is the solution of \eqref{ex1} when $b=\lambda(B)$.
    \end{enumerate}
\end{example}

\section{Critical angles between LISC cones}\label{sec:criang}

Let $P,Q\in\mathcal C(\mathbb V)$ be LISC cones represented by $P=G(K_\mathbb X)$ and $Q=H(K_{\mathbb Y})$, respectively. With these representations, the cosinus of the maximal angle between $P$ and $Q$, formulated in \eqref{maximal}, becomes
\begin{equation}\label{maxpol}
\left\{\begin{array}{ll}
\min&\displaystyle \langle Gx,Hy\rangle\\
\mbox{s.t.}&x\in K_{\mathbb X},\;\Vert Gx\Vert=1,\\
&y\in K_{\mathbb Y},\;\Vert Hy\Vert=1.
\end{array}\right.
\end{equation}
\begin{lemma}
 The pair $(\bx,\by)$ is a stationary point of \eqref{maxpol}  if and only if
\begin{equation}\label{stat1}
\left\{
    \renewcommand{\arraystretch}{1.5}
 \begin{array}{l}
\KX\ni \bar x\;\perp\; \left(G^\top H\by -\langle G\bx,H\by\rangle G^\top G\bx\right)\in\KX,\\
\KY\ni \bar y\;\perp\;\left(H^\top  G\bx-\langle G\bx,H\by\rangle H^\top H\by\right)\in\KY,\\
\Vert G\bx\Vert=1,\;\Vert H\by\Vert=1.
\end{array}
\right.
\end{equation}
\end{lemma}

\begin{proof}
Observe that for all $(x,y)$ feasible to problem~\eqref{maxpol}, the gradients of  $\Psi(x,y)=\langle Gx,Hy\rangle$, $\phi_1(x):=\Vert Gx\Vert-1$ and $\phi_2(y):=\Vert Hy\Vert-1$  are all well defined and  given by
    $$\nabla\Psi(x,y)=\left(G^\top Hy,H^\top G x\right),\;\nabla \phi_1(x)=\frac{G^\top G x}{\Vert G x\Vert}\;\mbox{and}\;\nabla \phi_2(y)=\frac{H^\top H y}{\Vert H y\Vert}.$$
It follows from~\eqref{crit2} that 
\begin{align*}
    0&=\inner{\bx}{\nabla_x \Psi(\bx,\by)}+\gamma_1\inner{\bx }{\nabla \phi_2(\bx)}=
\inner{G\bx}{H\by}+\gamma_1\norm{G\bx}=\inner{G\bx}{H\by}+\gamma_1\\
0&=\inner{\by}{\nabla_y \Psi(\bx,\by)}+\gamma_2\inner{\by }{\nabla \phi_2(\by)}=
\inner{G\bx}{H\by}+\gamma_2\norm{H\by}=\inner{G\bx}{H\by}+\gamma_2\\
\end{align*}
i.e., $\gamma_1=\gamma_2=-\inner{G\bx}{H\by}$.
The result then follows from~\eqref{crit2} by recalling that
$\KX$ and $\KY$ are self-dual, then $\KX^\ast=\KX$ and $\KY^\ast=\KY$.

\end{proof}

\begin{proposition}\label{prop:link(x,y)(u,v)}
If $(\bx,\by)$ is a stationary point (resp., solution) of \eqref{maxpol}, then $(G\bx,H\by)$ is a critical (resp., antipodal) pair of $(P,Q)$. Conversely, if $(u,v)$ is a critical (resp., antipodal) pair of $(P,Q)$, then there exists a stationary point (resp., solution) $(\bx,\by)$ of \eqref{maxpol} such that $u=G\bx$ and $v=H\by$. 
\end{proposition}
\begin{proof}

Let $(\bx,\by)$ be a stationary point of \eqref{maxpol}, and set
$u=G\bx$ and $v=H\by$. It follows from~\eqref{stat1}, and the fact that  $\KX^\ast=\KX$ and $\KY^\ast=\KY$, that
\begin{equation*}
\left\{
    \renewcommand{\arraystretch}{1.5}
 \begin{array}{l}
\KX\ni \bar x\;\perp\; G^\top \left(v -\langle u,v\rangle  u\right)\in\KX^\ast,\\
\KY\ni \bar y\;\perp\;H^\top  \left(u-\langle u,v\rangle  v\right)\in\KY^\ast,\\
\Vert u\Vert=1,\;\Vert v\Vert=1.
\end{array}
\right.
\end{equation*}
By definition of $P$ and $Q$, the conditions $\bx\in \KX$ and $\by\in \KY$ imply that $u=G\bx\in P$ and $v=H\by\in Q$, respectively. These facts and the last line of the above system mean that the first two lines of  the definition \eqref{crit} of critical angles of $(P,Q)$ are satisfied. Moreover, note that (i) $\bar x\;\perp\; G^\top \left(v -\langle u,v\rangle  u\right)$ is equivalent to
$G\bar x\;\perp\; v -\langle u,v\rangle  u$,   and (ii)
\begin{align*}
 G^\top \left(v -\langle u,v\rangle  u\right)\in\KX^\ast &\Leftrightarrow 
 \inner{G^\top \left(v -\langle u,v\rangle  u\right)}{x}\geq 0\; \forall x \in \KX,\\
 &\Leftrightarrow 
 \inner{ v -\langle u,v\rangle  u}{Gx}\geq 0\; \forall x \in \KX,\\
  &\Leftrightarrow 
 \inner{ v -\langle u,v\rangle  u}{p}\geq 0\; \forall p \in P,\\
 &\Leftrightarrow  v -\langle u,v\rangle  u \in P^\ast.
\end{align*}
Therefore, items (i) and (ii) assert that the first line in the above system is equivalent to $P\ni u \;\perp\; v -\langle u,v\rangle  u \in P^\ast$.
Analogously, the second line in the system reads as $Q\ni v \;\perp\; u -\langle u,v\rangle  v \in Q^\ast$.
Hence, $(u,v)$ satisfies \eqref{crit} which means that $(u,v)$ is a critical pair of $(P,Q)$. 

Conversely, let $(u,v)$ be a critical pair of $(P,Q)$. Then, $u\in P=G(\KX)$ and $v\in Q=H(\KX)$ which mean that there exist $\bx\in\KX$ and $\by\in\KX$ such that $u=G\bx$ and $v=H\by$. By replacing $(u,v)=(G\bx,H\by)$ in \eqref{crit} and by using items (i) and (ii), described above, we deduce that $(\bx,\by)$ satisfies \eqref{stat1}. Thus, $(\bx,\by)$ is a stationary point of \eqref{maxpol}.

Now, let $C_1\times C_2$ be the feasible set of \eqref{maxpol} where $C_1:=\{x:x\in\KX,\,\Vert Gx\Vert=1\}$ and $C_2:=\{y:y\in\KY,\,\Vert Hy\Vert=1\}$. It is straightforward to see that the feasible sets of \eqref{maximal} and \eqref{maxpol}  are related by
\begin{equation}\label{feaseq}
    \{(Gx,Hy):(x,y)\in C_1\times C_2\}=\{(u,v):u\in P,\,\Vert u\Vert=1,v\in Q,\,\Vert v\Vert=1\}.
\end{equation}
Hence, if $(\bx,\by)$ is a solution of \eqref{maxpol} then $\langle G\bx,H\by\rangle\leq \langle G x,H y\rangle$ for all $(x,y)\in C_1\times C_2$. From \eqref{feaseq} we deduce that $(\bar u,\bar v):=(G\bx,H\by)$ is a solution of \eqref{maximal}. Then, $(\bar u,\bar v)$ is an antipodal pair of $(P,Q)$. Conversely, if $(\bar u,\bar v)$ is an antipodal pair of $(P,Q)$ then there exist $(\bx,\by)\in C_1\times C_2$ such that $(G\bx,H \by)=(\bar u,\bar v)$ satisfying $\langle G\bx,H\by\rangle\leq \langle u, v\rangle$ for all  feasible point $(u,v)$ of \eqref{maximal}. From \eqref{feaseq} we conclude that $(\bx,\by)$ is a solution of \eqref{maxpol}.

\end{proof}

The numerical resolution of problem \eqref{maxpol} is not an easy task since it has as feasible set 
$$\{x\in K_\mathbb X:\Vert Gx\Vert=1\}\times  \{y\in K_\mathbb Y:\Vert Hy\Vert=1\},$$
  which is not convex due to its spherical conditions. Thus, we need to reformulate \eqref{maxpol} in such a way that its feasible set can be more conveniently manipulable.  Our reformulation will be a minimization problem with feasible set $\Delta_{\mathbb X,\mathbb Y}=\DX\times \Delta_{\mathbb Y}$. Recall that  $\DX$ and $\Delta_{\mathbb Y}$ are the spectraplexes of $\mathbb X$ and $\mathbb Y$ which are given respectively by
 \begin{equation*}
 \Delta_{\mathbb X}=\{x\in K_\mathbb X:\langle e_\mathbb X,x\rangle=1\}\quad\mbox{and}\quad \Delta_{\mathbb Y}= \{y\in K_\mathbb Y:\langle e_\mathbb Y,y\rangle=1\}.
 \end{equation*}
  The set $\Delta_{\mathbb X,\mathbb Y}$ turns out to be convex since we have hyperplanes instead of spheres. As $Gx\neq 0$ and $Hy\neq 0$ for all $(x,y)\in \Delta_{\mathbb X,\mathbb Y}$, due to Lemma~\ref{lem:lisc},   Proposition~\ref{prop:link(x,y)(u,v)} asserts that problem \eqref{maxpol} can be reformulated as~\eqref{maxfrac}, 
fitting thus into the structure of~\eqref{model}.
 We have that $\Phi$ (in~\eqref{maxfrac}) is differentiable in $\Delta_{\mathbb X,\mathbb Y}$. Its gradient is given $\nabla\Phi(x,y)=(\nabla_x\Phi(x,y),\nabla_y\Phi(x,y))$, where
 \begin{subequations}\label{gradphi}
 \begin{eqnarray}
\nabla_x\Phi(x,y)&=&\frac{G^\top H y}{\Vert Gx\Vert \Vert Hy\Vert}-\frac{\langle Gx,Hy\rangle}{\left\|Gx\right\| \left\|Hy\right\|}\frac{G^\top G x}{\Vert Gx\Vert^2},\\
\nabla_y\Phi(x,y)&=&\frac{H^\top G x}{\Vert Gx\Vert \Vert Hy\Vert}-\frac{\langle Gx,Hy\rangle}{\left\|Gx\right\| \left\|Hy\right\|}\frac{H^\top H y}{\Vert Hy\Vert^2}.
\end{eqnarray}
\end{subequations}

\begin{lemma}
 The pair $(\bx,\by)$ is a stationary point of \eqref{maxfrac}  if and only if
\begin{equation}\label{stat2}
\left\{
    \renewcommand{\arraystretch}{2.3}
 \begin{array}{l}
\displaystyle\KX\ni \bx\;\perp\; \left(\frac{G^\top H\by}{ \Vert H\by\Vert} -\frac{\langle G\bx,H\by\rangle}{\Vert G\bx\Vert \Vert H\by\Vert} \frac{G^\top G\bx}{\Vert G\bx\Vert}\right)\in\KX,\\
\displaystyle\KY\ni \bar y\;\perp\;\left(\frac{H^\top G\bx}{ \Vert G\bx\Vert} -\frac{\langle G\bx,H\by\rangle}{\Vert G\bx\Vert \Vert H\by\Vert} \frac{H^\top H\by}{\Vert H\by\Vert}\right)\in\KY,\\
\langle e_\mathbb X,\bx\rangle=1,\;\langle e_\mathbb Y,\by\rangle=1.
\end{array}
\right.
\end{equation}
\end{lemma}

\begin{proof}
    It follows from Definition\,\ref{def:stationary} and the equations in~\eqref{gradphi} that $(\bx,\by)$ is a stationary point of \eqref{maxfrac}  if, and only if, there exist $\gamma_1,\gamma_2\in\mathbb R$ such that
    \begin{align}
        &\displaystyle\KX\ni \bx\;\perp\; \left(\frac{G^\top H\by}{ \Vert G\bx\Vert \Vert H\by\Vert} -\frac{\langle G\bx,H\by\rangle}{\Vert G\bx\Vert \Vert H\by\Vert} \frac{G^\top G\bx}{\Vert G\bx\Vert^2}+\gamma_1 e_{\mathbb X}\right)\in\KX^\ast,\label{rel1}\\
&\displaystyle\KY\ni \bar y\;\perp\;\left(\frac{H^\top G\bx}{ \Vert G\bx\Vert \Vert H\by\Vert} -\frac{\langle G\bx,H\by\rangle}{\Vert G\bx\Vert \Vert H\by\Vert} \frac{H^\top H\by}{\Vert H\by\Vert^2}+\gamma_2 e_{\mathbb Y}\right)\in\KY^\ast,\label{rel2}\\
&\langle e_\mathbb X,\bx\rangle=1,\;\langle e_\mathbb Y,\by\rangle=1.\label{rel3}
    \end{align}
    Next we show that $\gamma_1=\gamma_2=0$. From the orthogonality conditions \eqref{rel1} and \eqref{rel2} we have that $\gamma_1\langle e_\mathbb X,\bx\rangle=0$ and $\gamma_2\langle e_\mathbb Y,\by\rangle=0$, respectively. Thus, because of \eqref{rel3} we deduce $\gamma_1=\gamma_2=0$. Finally, the system~\eqref{rel1}-\eqref{rel3} becomes~\eqref{stat2} by recalling that $\KX$ and $\KY$ are self-dual, and by removing w.l.o.g. $1/\norm{G\bx}$ from equation~\eqref{rel1} and $1/\norm{H\by}$ from equation~\eqref{rel2}.
\end{proof}

\begin{theorem}\label{propfrac}
If  $(\bx,\by)$ is a stationary point (resp., solution) of the problem \eqref{maxpol}, then  
$(\tx,\ty):=\left(\frac{\bx}{\langle e_\mathbb X,\bx\rangle},\frac{\by}{\langle e_\mathbb Y,\by\rangle}\right)$
is a stationary point (resp., solution) of \eqref{maxfrac}. Conversely, if $(\tx,\ty)$ is a stationary point (resp., solution) of \eqref{maxfrac}, then $(\bx,\by):=\left(\frac{\tx}{\Vert G\tx\Vert},\frac{\ty}{\Vert H\ty\Vert}\right)$
is a stationary point (resp., solution) of \eqref{maxpol}.
\end{theorem}
\begin{proof}
    Let $(\bx,\by)$ be a stationary point of \eqref{maxpol}. Then \eqref{stat1} holds. From Lemma~\ref{positive} we have that $\langle \eX,\bx\rangle>0$ and $\langle \eY,\by\rangle>0$. Thus, the pair $(\tx,\ty):=\left(\frac{\bx}{\langle e_\mathbb X,\bx\rangle},\frac{\by}{\langle e_\mathbb Y,\by\rangle}\right)$ is well-defined. Observe that because of the unit norm conditions in \eqref{stat1} we get $\Vert G\tx\Vert=\langle \eX,\bx\rangle^{-1}$ and  $\Vert H\ty\Vert=\langle \eY,\by\rangle^{-1}$. Then, $\bx=\frac{\tx}{\Vert G\tx\Vert}$ and $\by=\frac{\ty}{\Vert H\ty\Vert}$. By replacing these values in  \eqref{stat1} we deduce that $(\tx,\ty)$ satisfies \eqref{stat2}. Therefore, $(\tx,\ty)$ is a stationary point of \eqref{maxfrac}. The converse is proved analogously. In this case, when $(\tx,\ty)$ is a stationary point of \eqref{maxfrac}, $G\tx\neq0$ and $H\ty\neq 0$ thanks to Lemma\,\ref{lem:lisc}. Then, $(\bx,\by):=\left(\frac{\tx}{\Vert G\tx\Vert},\frac{\ty}{\Vert H\ty\Vert}\right)$ is well-defined, and it turns out to be a stationary point of \eqref{maxpol}. The details are omitted.

    Now, suppose that $(\bx,\by)$ is a solution of \eqref{maxpol}. Let $(x,y)$ be any feasible point of \eqref{maxfrac}. Then, $\left(\frac{x}{\Vert Gx\Vert},\frac{y}{\Vert Hy\Vert}\right)$ is well-defined and it is a feasible point of \eqref{maxpol}. Therefore,
    \begin{equation*}
    \inner{G\bx}{H\by}\leq\Phi(x,y),
    \end{equation*}
    for all feasible points $(x,y)$ of \eqref{maxfrac}. In particular, $(\tx,\ty):=\left(\frac{\bx}{\langle e_\mathbb X,\bx\rangle},\frac{\by}{\langle e_\mathbb Y,\by\rangle}\right)$ is a feasible point of \eqref{maxfrac} and  it attains the equality in the above displayed inequality. Thus $(\tx,\ty)$ is a solution of \eqref{maxfrac}. The converse is analogous.
    \end{proof}
    The connection between~\eqref{maxpol} and the problem of computing a critical angle between the cones $P=G(\KX)$ and $Q=H(\KY)$ are made explicit in the following corollary.
\begin{corollary}\label{coro-frac}
    If $(\tx,\ty)$ is a stationary point (resp., solution) of  \eqref{maxfrac}, then $(u,v):=\left(\frac{G\tx}{\Vert G\tx\Vert},\frac{H\ty}{\Vert H\ty\Vert}\right)$
    is a critical
(resp., antipodal) pair of $(P,Q)$.  Conversely, if $(u,v)$ is a critical (resp., antipodal) pair of $(P,Q)$, then there exists a stationary point $(\tx,\ty)$ of  \eqref{maxfrac} such that $(u,v)$ is given as above.
\end{corollary}
\begin{proof}
    Combine Theorem~\ref{propfrac} with Proposition~\ref{prop:link(x,y)(u,v)}.
\end{proof}
Hence, computing a critical angle between two cones $P=G(\KX)$ and $Q=H(\KY)$ amounts to computing a stationary point of the fractional programming problem~\eqref{maxfrac}. The next section is dedicated to this task.

\section{A sequential regularized partial linearization algorithm}\label{sec:alg}

In this section, we develop a numerical method to compute critical angles between a pair of LISC cones by computing stationary points of the fractional program \eqref{maxfrac}. By using the well-known approach introduced by Dinkelbach in \cite{D67}, we consider the following parametric program, for $\delta\in\mathbb R$:
\begin{equation}\label{bcp}
{\rm BCP}_\delta\;:\;
\begin{cases}
\min&f_{\delta}(x,y):=\langle Gx,Hy\rangle-\delta\left\|Gx\right\| \left\|Hy\right\|\\
\mbox{s.t.}&x\in K_{\mathbb X},\;\langle e_\mathbb X,x\rangle=1,\\
&y\in K_{\mathbb Y},\;\langle e_\mathbb Y,y\rangle=1.
\end{cases}
\end{equation}
\begin{remark}
We use the notation ${\rm BCP}_\delta$ since that program is constituted by biconvex functions (see \cite{GPK} for definitions and properties of biconvexity). Indeed, $\mathcal B(x,y):=\langle Gx,Hy\rangle$ is biconvex because it is bilinear. Furthermore, $\mathcal D(x,y):=\left\|Gx\right\| \left\|Hy\right\|$ is biconvex because its function at each variable is a multiple of the composition of the cartesian norm with a linear map which is a convex function.
\end{remark}
Observe that $f_{\delta}$ is differentiable at any $(x,y)$ such that $Gx\neq 0$ and $Hy\neq 0$. As discussed in the previous section, those conditions hold for every $(x,y)\in \Delta_{\mathbb X,\mathbb Y}$. Thus, $f_{\delta}$ is differentiable over the set  $\Delta_{\mathbb X,\mathbb Y}$ of feasible points of ${\rm BCP}_\delta$, and the gradient is given by
 $\nabla f_{\delta}(x,y)=\left(\nabla_x f_{\delta}(x,y),\nabla_y f_{\delta}(x,y)\right)$, where
 \begin{subequations}
 \begin{eqnarray}
&&\nabla_x f_{\delta}(x,y)=G^\top H y-\delta \Vert Gx\Vert^{-1}\Vert Hy\Vert G^\top G x,\label{gradx}\\
&&\nabla_y f_{\delta}(x,y)=H^\top G x-\delta \Vert Hy\Vert^{-1}\Vert Gx\Vert H^\top H y.\label{grady}
\end{eqnarray}
\end{subequations}

The next theorem states that a stationary point of \eqref{maxfrac} can be computed as a stationary point of ${\rm BCP}_\delta$ where $f_\delta$ vanishes.

\begin{theorem}\label{th:param}
The pair $(\bx,\by)$ is a stationary point (resp., solution) of ${\rm BCP}_\delta$ {\eqref{bcp}} with parameter $ \delta:= \frac{\langle G \bx,H \by\rangle}{\left\|G \bx\right\| \left\|H\by\right\|}$  if and only if $(\bx,\by)$ is a stationary point (resp., solution) of {\rm FP} {\eqref{maxfrac}}.
\end{theorem}
\begin{proof}
Let $(\bx,\by)$ be a stationary point of ${\rm BCP}_\delta$ {\eqref{bcp}}. From Definition\,\ref{def:stationary} and because $\KX$ and $\KY$ are self-dual, we get
\[
\left\{
\begin{array}{l}
\KX\ni \bar x\;\perp\;\left(G^\top H \by-\delta \Vert G\bx\Vert^{-1}\Vert H\by\Vert G^\top G \bx+\gamma_1e_{\mathbb{X}}\right)\in \KX,\\
\KY\ni \bar y\;\perp\;\left(H^\top G \bx-\delta \Vert H\by\Vert^{-1}\Vert G\bx\Vert H^\top H \by+\gamma_1e_{\mathbb{Y}}\right)\in \KY,\\
\langle e_\mathbb X,\bx\rangle=1,\;\langle e_\mathbb Y,\by\rangle=1.
\end{array}
\right.
\]
From the orthogonality conditions above we deduce that
$0=\inner{G\bx}{H\by} - \delta \norm{H\by}\norm{G\bx}+\gamma_1$
and $0=\inner{G\bx}{H\by} - \delta \norm{H\by}\norm{G\bx}+\gamma_2$. The definition of $\delta$ then implies that $\gamma_1=\gamma_2=0$. The above system then becomes (with $ \delta= \frac{\langle G \bx,H \by\rangle}{\left\|G \bx\right\| \left\|H\by\right\|}$)
\[
\left\{
\begin{array}{l}
\KX\ni \bar x\;\perp\;\left(G^\top H \by-\frac{\langle G \bx,H \by\rangle}{\left\|G \bx\right\|^2 }  G^\top G \bx\right)\in \KX,\\
\KY\ni \bar y\;\perp\;\left(H^\top G \bx-\frac{\langle G \bx,H \by\rangle}{ \left\|H\by\right\|^2}  H^\top H \by\right)\in \KY,\\
\langle e_\mathbb X,\bx\rangle=1,\;\langle e_\mathbb Y,\by\rangle=1.
\end{array}
\right.
\]

As $\bar x\;\perp\;\left(G^\top H \by-\frac{\langle G \bx,H \by\rangle}{\left\|G \bx\right\|^2 }  G^\top G \bx\right)\in \KX$ , it is equivalent to
$$\bar x\;\perp\;\frac{1}{\norm{H\by}}\left(G^\top H \by-\frac{\langle G \bx,H \by\rangle}{\left\|G \bx\right\|^2}  G^\top G \bx\right)\in \KX.$$ 
Analogously, we have $\bar y\;\perp\;\frac{1}{\norm{G\bx}}\left(H^\top G \bx-\frac{\langle G \bx,H \by\rangle}{ \left\|H\by\right\|^2}  H^\top H \by\right)\in \KY$. Then, we can deduce that $(\bx,\by)$ is a stationary point of ${\rm BCP}_\delta$ {\eqref{bcp}} with  $\delta= \frac{\langle G \bx,H \by\rangle}{\left\|G \bx\right\| \left\|H\by\right\|}$  if and only if $(\bx,\by)$ satisfies \eqref{stat2}. Thus, the equivalence follows. Suppose that $(\bx,\by)$ is a solution of ${\rm BCP}_\delta$ {\eqref{bcp}} with $\delta= \frac{\langle G \bx,H \by\rangle}{\left\|G \bx\right\| \left\|H\by\right\|}$. Then, $f_{\delta}(\bx,\by)=0$. Furthermore, for any feasible point $(x,y)$ of ${\rm BCP}_\delta$ {\eqref{bcp}} we get $f_{\delta}(x,y)\geq0$, which means 
\begin{equation}\label{ineqth}
\Phi(\bx,\by)=\delta\leq \Phi(x,y).
\end{equation}
Then, since problems ${\rm BCP}_\delta$ {\eqref{bcp}} and {\rm FP} {\eqref{maxfrac}} have the same feasible sets,  we conclude from \eqref{ineqth} that $(\bx,\by)$ is a solution of {\rm FP} {\eqref{maxfrac}}. The converse follows the same steps but in reverse order.

\end{proof}

\begin{corollary}\label{coro:2}
If $(\bx,\by)$ is a stationary point (resp., solution) of ${\rm BCP}_\delta$ {\eqref{bcp}} with parameter $ \delta:= \frac{\langle G \bx,H \by\rangle}{\left\|G \bx\right\| \left\|H\by\right\|}$, then $(u,v):=\left(\Vert G\bx\Vert^{-1}G\bx,\Vert H\by\Vert^{-1}H\by\right)$ is a critical (resp., antipodal) pair of $(P,Q)$.  Conversely, if $(u,v)$ is a critical (resp., antipodal) pair of $(P,Q)$, then there exists a stationary point $(\bx,\by)$ of ${\rm BCP}_\delta$ {\eqref{bcp}} with $ \delta:= \frac{\langle G \bx,H \by\rangle}{\left\|G \bx\right\| \left\|H\by\right\|}$  and such that $(u,v)=\left(\Vert G\bx\Vert^{-1}G\bx,\Vert H\by\Vert^{-1}H\by\right)$.
\end{corollary}
\begin{proof}
    It follows directly from Theorem~\ref{th:param} and Corollary~\ref{coro-frac}.
\end{proof}

Next, we formulate our algorithm to compute critical angles between a pair of LISC cones $(P,Q)=(G(K_\mathbb X), \,H(K_{\mathbb Y})).$
\begin{algorithm}[H]
\caption{Sequential Regularized Partial Linearization}\label{alg:main}
\begin{algorithmic}[1]
\Statex \Comment {\textbf{Step $0$}}
\State Choose $(x^0,y^0)\in \Delta_{\mathbb X,\mathbb Y}$, $\beta>0$, $0<\alpha<1$, $0<\rho<1$, and prox-parameters $\mu_1,\mu_2\geq0$.
\State Set $k:=0$.
\vspace{0.2cm}
\Statex \Comment {\textbf{Step $1$}}
\State Set
$$\delta_k:=\frac{\langle Gx^k,Hy^k\rangle}{\Vert G x^k\Vert \Vert H y^k\Vert}.$$
\vspace{0.2cm}
\Statex \Comment {\textbf{Step $2$}}
\State Let $L^k_1(x) := \left\langle Gx,Hy^k-\delta_k\Vert Gx^k\Vert^{-1}\Vert Hy^k\Vert Gx^k\right\rangle$.\newline 
Compute a solution $\tilde x^k$ to the convex  program
\begin{equation}\label{lin1}
\begin{array}{ll}
\min & L^k_1(x) + \frac{\mu_1}{2} \|x - x^k\|^2  \\
s.t.&x\in K_{\mathbb X},\;\langle e_\mathbb X,x\rangle=1.
\end{array}
\end{equation}

\State Let $L^k_2(y) :=  \left\langle Hy,G x^k-\delta_k\Vert G x^k\Vert\Vert Hy^k\Vert^{-1} Hy^k\right\rangle$. \newline Compute a solution $\tilde y^k$ to the convex program
 \begin{equation}\label{lin2}
\begin{array}{ll}
\min & L^k_2(y) + \frac{\mu_2}{2} \|y - y^k\|^2 \\
s.t.&y\in K_{\mathbb Y},\; \langle e_\mathbb Y,y\rangle=1. 
\end{array}
\end{equation}

\State Let $d_1^k:=\tilde x^k-x^k$ and $d_2^k:=\tilde y^k-y^k$.
\vspace{0.2cm}
\Statex \Comment {\textbf{Step $3$}}

\State If $L^k_1( d_1^k)=0$ and $L^k_2( d_2^k)=0$, terminate. \newline Otherwise, let $t_k:= \beta \rho^{\ell_k}$, where $\ell_k$ is the smallest nonnegative integer $\ell$ such that
$$\Phi(x^k+t^k d^k_1,y^k+t^k d^k_2) \leq \Phi(x^k,y^k)  + 
\alpha t_k \frac{L^k_1( d_1^k)+L^k_2( d_2^k)}{\Vert Gx^k\Vert\Vert Hy^k\Vert}
$$
Set $\left(x^{k+1},y^{k+1}\right):=(x^k,y^k)+t_k(d^k_1,d^k_2)$. Go to Step\,1.

\end{algorithmic}
\end{algorithm}

Some comments on Algorithm~\ref{alg:main} are in order.
\begin{enumerate}[itemsep=2pt, topsep=2pt, label=\roman*), leftmargin=1.75cm, start=1] 
    \item The prox-parameters $\mu_1$ and $\mu_2$ can take arbitrary non-negative values. The choice  $\mu_1=\mu_2=0$ is covered by the convergence analysis given below  because $\Delta_{\mathbb{X},\mathbb{Y}}$ of~\eqref{maxfrac} is compact (see Proposition\,\ref{compact}). However, our numerical experiments suggest that choosing the prox-parameters to be strictly positive is advantageous: it regularizes the iterative process in the sense that the algorithm performs fewer iterations to compute a  stationary point within a given tolerance. 
        \item We highlight that Step 2 can be performed in parallel as the two subproblems are independent of each other. 
    \item For LISC cones, the interest of dealing with~\eqref{maxfrac} instead of~\eqref{maximal} is now evident: the cones $\KX$ and $\KY$ are in general more structured, permitting \eqref{lin1} and \eqref{lin2} to have straightforward solutions in many settings of interest (c.f. Subsection~\ref{sec:proj}). If the considered cones are not LISC, then $G$ and $H$ are the identities operators in the corresponding spaces $\mathbb X$ and $\mathbb Y$, and thus $P= K_{\mathbb X}$ and $Q= K_{\mathbb Y}$. Algorithm~\ref{alg:main} is still applicable:  for applications in which the subproblems do not have explicit solutions, we can use of off-the-shelf (linear, quadratic, conic, etc.) solvers to compute $\tilde x^k$ and $\tilde y^k$.
    \item In Step 3, we perform the standard Armijo line search rule. Indeed,  the term $\frac{L^k_1( d_1^k)+L^k_2( d_2^k)}{\Vert Gx^k\Vert\Vert Hy^k\Vert}$ coincides with $ \left\langle \nabla \Phi(x^k,y^k), d^k\right\rangle$. As we show in Proposition~\ref{prop:defined} below, $d^k$ is a descend direction for $\Phi$ at $(x^k,y^k)$.
\end{enumerate}

\begin{lemma}\label{lem:neg}
We have that 
\begin{equation}\label{ineqL}
L_1^k(d_1^k)\leq -\frac{\mu_1}{2}\Vert d_1^k\Vert^2\quad\mbox{and}\quad L_2^k(d_2^k)\leq -\frac{\mu_2}{2}\Vert d_2^k\Vert^2.
\end{equation}
Furthermore, if $x^k$ is not a solution to \eqref{lin1} or $y^k$ is not a solution to \eqref{lin2}, then 
\begin{equation}\label{ineqL12}
L_1^k(d_1^k)+L_2^k(d_2^k) < 0.
\end{equation}
\end{lemma}
\begin{proof}
Since $\tilde x^k$ solves \eqref{lin1} and $x^k$ is feasible, we have that $L^k_1\left(\tilde x^k\right)+\frac{\mu_1}{2}\Vert \tilde x^k-x^k\Vert^2 \leq L^k_1\left( x^k\right)$, with strict inequality if $x^k$ is not a solution to \eqref{lin1}. 
As $L_1^k\left(d_1^k\right)=L^k_1\left(\tilde x^k\right)- L^k_1\left( x^k\right)$, we conclude that $L_1^k\left(d_1^k\right) \leq -\frac{\mu_1}{2}\Vert d_1^k\Vert^2$.  Analogously, we get $L_2^k(d_2^k)\leq -\frac{\mu_2}{2}\Vert d_2^k\Vert^2$.  Suppose that $x^k$ is not a solution to \eqref{lin1} or $y^k$ is not a solution to \eqref{lin2}. Then, one of the inequalities in \eqref{ineqL} holds strictly. It implies,  
$L_1^k(d_1^k)+L_2^k(d_2^k)< -\frac{\mu_1}{2}\Vert d_1^k\Vert^2-\frac{\mu_2}{2}\Vert d_2^k\Vert^2\leq0.$ 
\end{proof}
\begin{proposition}\label{prop:defined}
Suppose that $x^k$ is not a solution to \eqref{lin1} or $y^k$ is not a solution to \eqref{lin2}. Then, the vector $\left(d_1^k,d_2^k\right)$ is a descent direction for the fractional function $\displaystyle\Phi$ (given in~\eqref{maxfrac}) at $(x^k,y^k)$.
\end{proposition}
\begin{proof}
We must prove that
\begin{equation}\label{descent}
\langle \nabla\Phi(x^k,y^k), (d_1^k,d_2^k)\rangle=\langle \nabla_x\Phi(x^k,y^k), d_1^k\rangle+\langle \nabla_y\Phi(x^k,y^k), d_2^k\rangle<0.
\end{equation}
From \eqref{gradphi}, we obtain,
\begin{eqnarray*}
\nabla_x\Phi(x^k,y^k)&=&\frac{G^\top H y^k}{\Vert Gx^k\Vert \Vert Hy^k\Vert}-\delta_k\frac{G^\top G x^k}{\Vert Gx^k\Vert^2},\\
\nabla_y\Phi(x^k,y^k)&=&\frac{H^\top G x^k}{\Vert Gx^k\Vert \Vert Hy^k\Vert}-\delta_k\frac{H^\top H y^k}{\Vert Hy^k\Vert^2}\\
\end{eqnarray*}
Then,
\begin{eqnarray*}
\langle \nabla_x\Phi(x^k,y^k), d_1^k\rangle&=&\Vert Gx^k\Vert^{-1}\Vert Hy^k\Vert^{-1}L^k_1( d_1^k),\\
\langle \nabla_y\Phi(x^k,y^k), d_2^k\rangle&=&\Vert Gx^k\Vert^{-1}\Vert Hy^k\Vert^{-1}L^k_2( d_2^k).
\end{eqnarray*}
Since $x^k$ is not a solution to \eqref{lin1} or $y^k$ is not a solution to \eqref{lin2}, we can use \eqref{ineqL12} to conclude \eqref{descent}.

\end{proof}

The following proposition explains why the stop criterium of Algorithm\,\ref{alg:main} allow us to find solutions for \eqref{lin1} and \eqref{lin2}.
\begin{proposition}\label{prop:stop}
    If $L_1^k(d_1^k)=0$, then $x^k$ solves \eqref{lin1}. Moreover, If $L_2^k(d_2^k)=0$, then $y^k$ solves \eqref{lin2}.
\end{proposition}
\begin{proof}
    Suppose that $L_1^k(d_1^k)=0$. Then, $L_1^k(\tilde x^k)-L_1^k(x^k)=0$. If $\mu_1=0$, the regularization term in \eqref{lin1} vanishes. Hence, because of $L_1^k(\tilde x^k)=L_1^k(x^k)$ we conclude that $x^k$ solves \eqref{lin1}.  If $\mu_1>0$, from \eqref{ineqL} we deduce that $d_1^k=\tilde x^k-x^k=0$. Then, $x^k=\tilde x^k$ which implies that $x^k$ solves \eqref{lin1}. The proof of the second part is analogous.
\end{proof}

Now, the following two results show that Algorithm\,\ref{alg:main} find stationary points of problem \eqref{maxfrac}. 

\begin{proposition}\label{propstop}
Let $(\bar x,\bar y)\in \Delta_{\mathbb X,\mathbb Y}$ and set $\bar\delta:=\frac{\langle G\bar x,H\bar y\rangle}{\Vert G\bar x\Vert \Vert H\bar y\Vert}$. Then, $(\bar x,\bar y)$ is a stationary point of \eqref{maxfrac} if, and only if, $\bar x$ and $\bar y$ are solutions of the respective programs:
\begin{equation}\label{linbar1}
\begin{array}{ll}
\min & \left\langle Gx,H\bar y-\bar\delta\Vert G\bar x\Vert^{-1}\Vert H\bar y\Vert G\bar x\right\rangle + \frac{\mu_1}{2} \| x -\bx \|^2\\
s.t.&x\in K_{\mathbb X},\;\langle e_\mathbb X,x\rangle=1,
\end{array}
\end{equation}
and
 \begin{equation}\label{linbar2}
\begin{array}{ll}
\min & \left\langle Hy,G\bar x-\bar \delta\Vert G\bar x\Vert\Vert H\bar y\Vert^{-1} H\bar y\right\rangle + \frac{\mu_2}{2} \| y -\by \|^2\\
s.t.&y\in K_{\mathbb Y},\; \langle e_\mathbb Y,y\rangle=1.
\end{array}
\end{equation}

\end{proposition}
\begin{proof}
The pair $(\bar x,\bar y)$ is a stationary point of \eqref{maxfrac} if, and only if, $\inner{\nabla f_{\bar \delta}(\bx,\by)}{((x,y)-(\bx,\by)}\geq 0$ for all $(x,y) \in \Delta_{\mathbb X,\mathbb Y}$, which is equivalent to
\[
\left\{
\begin{array}{c}
\inner{\nabla_x f_{\bar \delta}(\bx,\by)}{x-\bx}\geq 0 \quad \forall x\in K_{\mathbb X}\; \mbox{ s.t. }\;\langle e_\mathbb X,x\rangle=1 \\
\inner{\nabla_y f_{\bar \delta}(\bx,\by)}{y-\by}\geq 0 \quad \forall y\in K_{\mathbb Y}\; \mbox{ s.t. }\;\langle e_\mathbb Y,y\rangle=1 .
\end{array}\right.
\]
In view of~\eqref{gradx} and~\eqref{grady}, the above system is equivalent to saying that $\bx$ and $\by$ are solutions to~\eqref{linbar1} and~\eqref{linbar2}, respectively.
\end{proof}

\begin{theorem}\label{th:conv}
Every accumulation point $\left(\bar{x},\bar{y}\right)$ of the sequence $\left\{\left(x^k,y^k\right)\right\}$, generated by Algorithm\,\ref{alg:main}, is a stationary point of problem \eqref{maxfrac}. 
\end{theorem}
\begin{proof}
Let $\left(\bar{x},\bar{y}\right)$ be an accumulation point of $\left\{\left(x^k,y^k\right)\right\}$ (which is assured to exist because $\Delta_{\mathbb{X},\mathbb{Y}}$ is compact, see Proposition\,\ref{compact}). By passing to a subsequence if necessary, we may suppose that $(x^k,y^k)\to (\bar x,\bar y)$ and $(\tilde x^k,\tilde y^k)\to(\tilde x,\tilde y)$, as $k\to\infty$, where $(\tilde x,\tilde y)\in \Delta_{\mathbb{X},\mathbb{Y}}$. Let $x\in K_{\mathbb X}$ s.t. $\langle e_\mathbb X,x\rangle=1$ be arbitrary. Because $\tilde x^k$ solves \eqref{lin1}, we have that 
\begin{align}\label{solbar}
& \left\langle G\tilde x^k,Hy^k-\delta_k\Vert Gx^k\Vert^{-1}\Vert Hy^k\Vert Gx^k\right\rangle +\frac{\mu_1}{2}\norm{\tilde x^k-x^k}^2\\
&\qquad \leq
\left\langle Gx,Hy^k-\delta_k\Vert Gx^k\Vert^{-1}\Vert Hy^k\Vert Gx^k\right\rangle+\frac{\mu_1}{2}\norm{x-x^k}^2.\nn
\end{align}
By letting $k\to\infty$ on \eqref{solbar}, we obtain 
\begin{align*}
& \left\langle G\tx,H\by-\bar \delta\Vert G\bx\Vert^{-1}\Vert H\by\Vert G\bx\right\rangle +\frac{\mu_1}{2}\norm{\tx-\bx}^2\\
&\qquad \leq
\left\langle Gx,H\by-\bar\delta\Vert G\bx\Vert^{-1}\Vert H\by\Vert G\bx\right\rangle+\frac{\mu_1}{2}\norm{x-\bx}^2.\nn
\end{align*}
where 
\begin{equation*}
\bar\delta:=\frac{\langle G\bar x,H\bar y\rangle}{\Vert G\bar x\Vert \Vert H\bar y\Vert}.
\end{equation*}
Thus, we deduce that $\tilde x$ is a solution to~\eqref{linbar1}.
On the other hand, from standard arguments on the Armijo rule we have that $\left\langle\nabla\Phi(x^k,y^k),(d_1^k,d_2^k)\right\rangle\to 0$ as $k\to\infty$. In Proposition\,\ref{prop:defined}, we have shown that $\left\langle\nabla_x\Phi(x^k,y^k),d_1^k\right\rangle\leq0$ and $\left\langle\nabla_y\Phi(x^k,y^k),d_2^k\right\rangle\leq0$. Then, 
\begin{eqnarray}
&&\lim_{k\to\infty}\left\langle\nabla_x\Phi(x^k,y^k),d_1^k\right\rangle=\left\langle\nabla_x\Phi(\bar x,\bar y),\tilde x-\bar x\right\rangle=0,\label{limitx}\\
&&\lim_{k\to\infty}\left\langle\nabla_y\Phi(x^k,y^k),d_2^k\right\rangle=\left\langle\nabla_y\Phi(\bar x,\bar y),\tilde y-\bar y\right\rangle=0.\label{limity}
\end{eqnarray}
Since $\left\langle\nabla_x\Phi(\bar x,\bar y),\tilde x-\bar x\right\rangle=\Vert G\bar x\Vert^{-1}\Vert H\bar y\Vert^{-1}\bar L_1( \tilde x-\bar x)$, from \eqref{limitx} we conclude
$$\bar L_1(\tilde x)-\bar L_1(\bar x)=\bar L_1( \tilde x-\bar x)=0.$$
Thus, from Proposition\,\ref{prop:stop} we conclude that not only $\tx$ but also $\bar x$ is solution of \eqref{linbar1}. In the same way, we can prove that both $\ty$ and $\by$ are solutions of \eqref{linbar2}. (We highlight that if $\mu_1,\mu_2>0$ in Algorithm~\ref{alg:main}, then $\bx=\tx$ and $\by=\ty$ due to strongly convexity of the objective functions.)
Therefore, from Proposition\,\ref{propstop} we conclude that $(\bar x,\bar y)$ is a stationary point of \eqref{maxfrac}.
\end{proof}

Finally, it follows from Theorem~\ref{th:conv} and Corollary~\ref{coro:2} that
Algorithm~\ref{alg:main} asymptotically computes a critical pair 
$(u,v)=\left(\Vert G\bx\Vert^{-1}G\bx,\Vert H\by\Vert^{-1}H\by\right)$ of $(P,Q)$.

\section{Computational experiments}\label{sec:comp}
This section reports some numerical experiments for computing critical angles between pairs of LISC cones.
For each class of test problems considered below, we describe how the subproblems~\eqref{lin1} and \eqref{lin2}  take a form that can be efficiently solved.

Indeed, for the pair of LISC cones $(P,Q)=(G(K_\mathbb X),H(K_{\mathbb Y}))$, the optimization problems in Step 2 of Algorithm~\ref{alg:main} are equivalent to the projection problems onto 
\begin{equation*}\label{onto}
\Delta_{\mathbb X}=\{x\in\KX:\langle \eX,x\rangle=1\}\quad\mbox{and}\quad \Delta_{\mathbb Y}=\{y\in\KY:\langle \eY,y\rangle=1\},
\end{equation*}
when $\mu_1>0$ and $\mu_2>0$, respectively (if $\mu_1=\mu_2=0$, then subproblems~\eqref{lin1} and \eqref{lin2} have straightforward solutions). Let us focus on problem~\eqref{lin1}, a similar reasoning applies to~\eqref{lin2}. After some algebraic manipulation on the cost function of \eqref{lin1}, we can obtain the solution of~\eqref{lin1} by solving the following projection problem:
\begin{equation}\label{eq:proy}
\begin{array}{ll}
\min &  \frac{1}{2} \Big\|x - \big(x^k-\frac{c_k}{\mu_1}\big) \Big\|^2  \\
\mbox{s.t.}& x\in \Delta_{\mathbb X}
\end{array}\quad \mbox{with}\quad 
c_k :=  G^T(Hy^k-\delta_k\Vert Gx^k\Vert^{-1}\Vert Hy^k\Vert Gx^k),
\end{equation}

As benchmark for our analysis, we solved the nonlinear program~\eqref{maxfrac} by using \ipopt~\cite{WaBi}
distributed in the OPTI toolbox~\cite{CuWi}. We use default values for the solver parameters, except for the maximum number of iterations, which is incremented to $5000$ and the Absolute Function Tolerance increased to $10^{-6}$.
The numerical experiments were performed on an Intel Core i7 clocked at 2 GHz (32 GB RAM). 

The performance of the two solvers is measured by considering $10^3$ initial points
and by reporting in the tables below the following values:
\begin{itemize}
    \item $\theta_b(P,Q)$: the best critical angle, which we compute as the cosine inverse of the best objective function value of~\eqref{maxfrac} obtained by one solver with $10^3$ initial points.
    \item it$_b$, it$_a$, it$_w$: minimum, average and maximum number of iterations required by the algorithm to compute a stationary point.
    \item CPU$_b$, CPU$_a$, CPU$_w$: minimum, average and maximum CPU time in seconds required by the algorithm to compute a stationary point. 
\end{itemize}

In practice, we stop Algorithm~{\ref{alg:main}} when $|L^k_1(d^k_1)|\leq \epsilon_1$ and $|L^k_2(d^k_2)|\leq \epsilon_2$, with $\epsilon_1,\epsilon_2$ being small positive values. Additionally, we require that the $\delta_k - \delta_{k+5} \leq \epsilon_3$ for a given small $\epsilon_3>0$, that is, the algorithm runs until the objective function value of~{\eqref{maxfrac}} stabilizes over the last five iterations. We also set the maximum number of iterations to $5000$.

\subsection{Critical angles between two polyhedral cones}

As in~\cite{SeSo1}, we consider the computation of critical angles of a pair 
\begin{equation*}\label{polyhedral}
    (P,Q) = (G(\R_+^p),H(\R_+^q))
\end{equation*}
of polyhedral cones in $\R^n$, where $G$ and $H$ are matrices in $\R^{n \times p}$ and $\R^{n \times q}$, respectively, such that $P$ and $Q$ are LISC cones.
For $\mu_1,\mu_2>0$ and because of \eqref{eq:proy}, we can obtain the solution of~\eqref{lin1} by solving the following projection onto the simplex problem:
\begin{equation*}\label{eq:c}
\begin{array}{ll}
\min &  \frac{1}{2} \Big\|x - \big(x^k-\frac{c_k}{\mu_1}\big) \Big\|^2  \\
\mbox{s.t.}&x\geq 0,\\
&\langle{\bf 1}_p,x\rangle=1.
\end{array}
\end{equation*}
The resolution of this task is explained in Example\,\ref{ex:pyeccion} item \ref{exa}. It can be efficiently performed by the specialized algorithm of \cite{Cond}.

We compute critical angles between the nonnegative orthant $P=\R^n_+$ and the Schur cone 
\begin{equation*}
    Q = \Bigg\{x \in \R^n: \sum_{i=1}^k x_i \geq 0, \mbox{ for } k \in \{1,...,n-1\} \mbox{ and } x_1+...+x_n = 0\Bigg\}.
\end{equation*}
We consider $G = I_n$ and the columns of the matrix $H$ are in the following form 
\begin{equation*}
  h_1 = \frac{1}{\sqrt{2}},
\begin{pmatrix}
  1,-1,0,\ldots,0  
\end{pmatrix}^T, \ldots
 h_{n-1} = \frac{1}{\sqrt{2}},
\begin{pmatrix}
  0,\ldots,0,1,-1  
\end{pmatrix}^T.
\end{equation*}

We generate several instances by considering $n \in \{5,20, 50, 100, 500,700,1000\}$ and run Algorithm~\ref{alg:main} and \ipopt~with $10^3$ random initial points in the simplex.
For these test problems, we stop Algorithm~\ref{alg:main} when the stopping criteria are satisfied with $\epsilon_1 = 10^{-6}$, $\epsilon_2 = 10^{-6}$ and $\epsilon_3 = 10^{-5}$.
After some tuning, the prox-parameter values were chosen as $\mu_1 = 0.01$ and $\mu_2 = 2.6$.  

The performance of Algorithm~\ref{alg:main} and \ipopt~for solving these test problems is presented in Tables~\ref{tab:poly_rspl} and~\ref{tab:poly_ip}, respectively. 
This first set of numerical experiments demonstrates that, unlike \ipopt, the proposed algorithm can calculate a critical angle within the maximum number of iterations for all the instances tested. As expected, the required CPU time increases with the dimension of the problems, but it typically remains in the tens of seconds, even for the largest instance. Furthermore, we observe that, for all but two cases (i.e., $n = 100$ and $n = 200$), the best optimal value computed by our algorithm is larger than the one computed by {\tt IPOPT}, as evidenced by values in the first columns of the corresponding tables. The box-plots in Fig.~\ref{fig:poly} illustrate the distribution of the critical angles across the different initial points with a digital precision of $10^{-6}$. We can see that the interquartile ranges (IQRs) across the box-plots are relatively narrow for both the solvers, indicating consistent optimization performance. For $n\geq 400$, our algorithm can converge to better quality points (i.e., corresponding to a large objective function value) than \ipopt~for almost all the initial points, confirming its good performance.

\begin{table}[H]
\centering
\begin{tabular}{|l|ccccccc|}
\hline
n    & $\theta_b(P,Q)$ & it$_{b}$ & it$_{a}$ & it$_{w}$ & CPU$_b$ & CPU$_a$ & CPU$_w$  \\ 
\hline
5 & 8.5242e-01 $\pi$ & 13 & 88.64 & 1101 & 1.64e-04 & 1.05e-03 & 1.11e-02  \\ 
20 & 9.2821e-01 $\pi$ & 48 & 262.84 & 1325 & 7.03e-04 & 3.63e-03 & 1.74e-02  \\ 
50 & 9.5478e-01 $\pi$ & 113 & 519.89 & 1969 & 2.40e-03 & 1.10e-02 & 4.13e-02  \\ 
100 & 9.6794e-01 $\pi$ & 195 & 1120.44 & 4227 & 8.44e-03 & 4.80e-02 & 1.76e-01  \\ 
200 & 9.7013e-01 $\pi$ & 240 & 1279.29 & 3216 & 5.56e-02 & 3.01e-01 & 1.26e+00  \\ 
400 & 9.6683e-01 $\pi$ & 347 & 1257.72 & 2899 & 1.89e-01 & 7.21e-01 & 3.54e+00  \\ 
500 & 9.6689e-01 $\pi$ & 446 & 1236.44 & 2367 & 3.28e-01 & 1.00e+00 & 2.20e+00  \\ 
700 & 9.6626e-01 $\pi$ & 494 & 1174.94 & 2403 & 8.12e-01 & 2.22e+00 & 6.05e+00  \\ 
1000 & 9.6360e-01 $\pi$ & 561 & 1105.09 & 2009 & 3.02e+00 & 6.38e+00 & 1.20e+01  \\ 
\hline
\end{tabular}
\caption{Performance of Algorithm~\ref{alg:main} for computing a critical angle between two polyhedral cones with $10^3$ initial points.}\label{tab:poly_rspl}
\end{table}
\begin{table}[H]
\centering
\begin{tabular}{|l|ccccccc|}
\hline
n    & $\theta_b(P,Q)$ & it$_{b}$ & it$_{a}$ & it$_{w}$ & CPU$_b$ & CPU$_a$ & CPU$_w$  \\ 
\hline
5 & 8.5242e-01 $\pi$ & 11 & 19.63 & 79 & 2.23e-02 & 3.01e-02 & 6.97e-02  \\ 
20 & 9.2822e-01 $\pi$ & 60 & 207.65 & 1665 & 6.11e-02 & 1.96e-01 & 1.56e+00  \\ 
50 & 9.5483e-01 $\pi$ & 194 & 1733.40 & 5000 & 2.90e-01 & 2.36e+00 & 1.08e+01  \\ 
100 & 9.6812e-01 $\pi$ & 729 & 4517.13 & 5000 & 1.21e+00 & 9.29e+00 & 2.49e+01  \\ 
200 & 9.7352e-01 $\pi$ & 5000 & 5000.00 & 5000 & 1.35e+01 & 2.01e+01 & 5.83e+01  \\ 
400 & 9.5666e-01 $\pi$ & 5000 & 5000.00 & 5000 & 5.23e+01 & 5.46e+01 & 8.27e+01  \\ 
500 & 9.4717e-01 $\pi$ & 5000 & 5000.00 & 5000 & 7.85e+01 & 8.82e+01 & 1.36e+02  \\ 
700 & 9.5286e-01 $\pi$ & 5000 & 5000.00 & 5000 & 1.66e+02 & 1.85e+02 & 2.61e+02  \\ 
1000 & 9.3832e-01 $\pi$ & 5000 & 5000.00 & 5000 & 4.25e+02 & 4.56e+02 & 6.32e+02  \\ 
\hline
\end{tabular}
\caption{Performance of \ipopt~for computing a critical angle between two polyhedral cones with $10^3$ initial points.}\label{tab:poly_ip}
\end{table}

\begin{figure}[htb]
    \centering
    \includegraphics[width=0.7\textwidth, trim={5.3cm 9.5cm 5.3cm 10.5cm}]{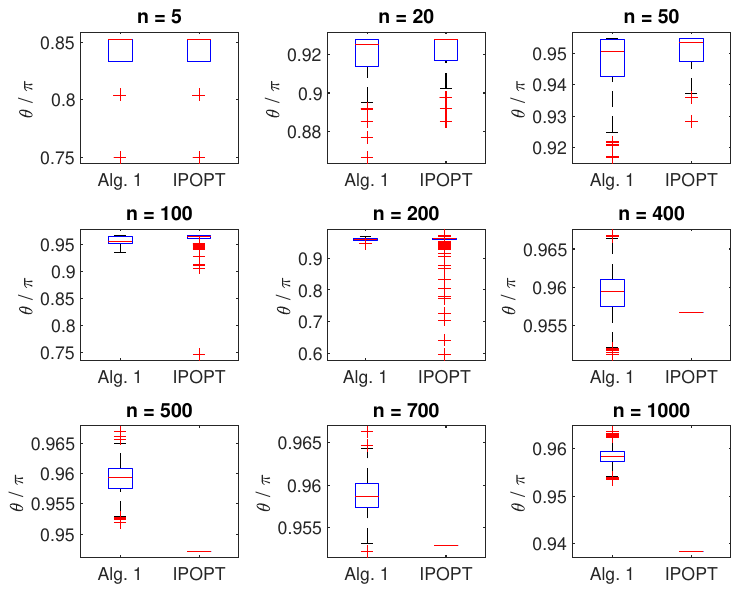}
    \caption{Box-plots of the critical angles between two polyhedral cones computed by Algorithm\,\ref{alg:main} and \ipopt~with $10^3$ initial points.}
    \label{fig:poly}
\end{figure}

In terms of CPU time, our approach is significantly faster than \ipopt: compare the three last columns of Tables~\ref{tab:poly_rspl} and~\ref{tab:poly_ip}. More precisely, \ipopt~took about 226.5 hours (approximately, 9 days) to solve all the $9000$ instances, while Algorithm~\ref{alg:main} took only 2.9 hours.

In \cite[Example\,7.4]{SeSo1}, all the critical angles for the case $n=5$ were computed and are shown in Table~\ref{tab:dist}. It is interesting to compare the number of times that our algorithm and \ipopt~compute the maximum critical angle. To this end, we ran the two algorithms with $10^3$ random initial points in the standard simplex, and in Table\,\ref{tab:dist}, we report (in percentage) the distribution of the critical angles computed. The results show that both Algorithm~\ref{alg:main} and \ipopt~compute the maximum angle for more than $60 \%$ of the total runs, and their performance is comparable in this case.
\begin{table}[H]
\begin{center}
\begin{tabular}{|l|c|c|c|c|c|c|c|c|c|c|c|c|c|c|c|c|c|c|c|c|c|c|}
\hline
Solver $/$ Angle& {\scriptsize $0.6476\,\pi$; $0.6667\,\pi$; $0.6959\,\pi$; $0.7180\,\pi$} &{\scriptsize $0.7500\,\pi$} & {\scriptsize $0.7820\,\pi$  }& {\scriptsize $0.8041\,\pi$ } & {\scriptsize $0.8333\,\pi$ }& {\scriptsize $0.8524\,\pi$ } \\
\hline
Algorithm~\ref{alg:main}&-&1.8\% & - & 12.6 \%  &21.6 \%  & 64 \%\\
\ipopt &-&1.7\% & - & 13.9 \%  &21.1 \%  & 63.3 \%\\
\hline
\end{tabular}
\caption{\small Percentage of critical angles between the nonnegative orthant  and the Schur cone in $\mathbb{R}^5$ found by Algorithm\,\ref{alg:main} with $10^3$ initial points.}\label{tab:dist}
\end{center}

\end{table}

\subsection{Critical angles between two ellipsoidal cones}
In this section, we evaluate the performance of Algorithm\,\ref{alg:main} for computing critical angles between a pair of ellipsoidal cones:
$$P=\{(\xi,t) \in \R^{n-1}\times \mathbb R: \sqrt{\langle \xi, A \xi} \rangle \leq t \}\quad\mbox{and}\quad Q=\{(\xi,t) \in \R^{n-1}\times \mathbb R: \sqrt{\langle \xi, B \xi} \rangle \leq t \},$$
where $A$ and $B$ are symmetric positive definite matrices of order $n$. Recall that $P$ and $Q$ are LISC cones which are represented by
\begin{align*}\label{eq:Lor}
P=G(\mathcal L_n^+)\quad\mbox{and}\quad Q=H(\mathcal L_n^+),
\end{align*}
where $\mathcal L_n^+$ is the Lorentz cone in $\mathbb R^n$ and $G,H:\mathbb R^n\to\mathbb R^n$ are linear maps given by
$$G(\xi,t)=(A^{-1/2}\xi,t)\quad\mbox{and}\quad H(\xi,t)=(B^{-1/2}\xi,t),$$
for all $(\xi,t) \in \R^{n-1}\times \mathbb R.$
As before, the subproblems in Step 2 of Algorithm~\ref{alg:main} are reformulated (when $\mu_1,\mu_2>0$) as orthogonal projection problems \eqref{eq:proy}. In this case, from Example\,\ref{ex:pyeccion} item\,\ref{exb}, we have that the solution to problem \eqref{eq:proy} is 
$$\wx^k=\left(\frac{\tilde b^k}{\max\{1,\Vert \tilde  b^k\Vert\}},1\right),$$
where $\tilde b^k \in \mathbb R^{n-1}$ is the vector consisting of the first $n-1$ elements of $(x^k-\frac{c_k}{\mu_1})$ and $c_k$ given in~\eqref{eq:proy}. 
For this set of test problems, we generate positive definite matrices $A$ and $B$  as follows
$$A = C+ n I \mbox{ and } B = D + nI$$ 
where $C$ and $D$ are randomly generated sparse symmetric matrices with elements normally distributed, with mean $0$ and variance $1$. We generate several instances by considering $n \in\{5,20,50,100,500,700,1000\}$ and density $d=0.5$. The tolerances for the algorithm stopping test are $\epsilon_1=10^{-6}$, $\epsilon_2=10^{-6}$ and $\epsilon_3=10^{-7}$.
We consider a maximum number of iterations equal to $5000$, and the value of the prox-parameters are chosen as $\mu_1 = 0.005$ and $\mu_2 = 0.005$.  
Finally, the initial point $(x^0,y^0)$ is set as $x^0=(\xi^0,1)$ and $y^0=(\nu^0,1)$, where $\xi^0$ and $\nu^0$ are random points in the Euclidean unit ball in $\R^{n-1}$.

We report in Tables~\ref{tab:elli_rspl} and \ref{tab:elli_ip} the performance of Algorithm~\ref{alg:main} and \ipopt, respectively. 
The computed critical angles by both algorithms are found to be very close for most instances. However, for the last two instances, \ipopt~manages to find a slightly better solution. It should be noted that \ipopt~achieves this result for only one initial point out of $10^3$ in the case of $n = 700$ and $n = 1000$, as depicted in Fig~\ref{fig:elli}.

One significant advantage of the proposed algorithm is that it requires significantly less computational effort compared to \ipopt. In fact, our algorithm is observed to be even $10^2$ times faster for the largest instances. For instance, \ipopt~took about 379 hours (15 days) to solve all the $9000$ instances, while Algorithm~\ref{alg:main} took only 4.5 hours. This substantial reduction in computation time allows for additional runs of our algorithm to explore a larger range of initial points, which could potentially yield improved critical angles.

\begin{table}[htb]
\centering
\begin{tabular}{|l|ccccccc|}
\hline
n    & $\theta_b(P,Q)$ & it$_{b}$ & it$_{a}$ & it$_{w}$ & CPU$_b$ & CPU$_a$ & CPU$_w$  \\ 
\hline
5 & 5.6950e-01 $\pi$ & 9 & 10.43 & 13 & 1.86e-04 & 3.93e-04 & 2.98e-02  \\ 
20 & 1.5928e-01 $\pi$ & 92 & 286.99 & 825 & 1.76e-03 & 5.68e-03 & 1.67e-02  \\ 
50 & 9.7918e-02 $\pi$ & 101 & 266.82 & 732 & 5.01e-03 & 1.16e-02 & 2.91e-02  \\ 
100 & 6.7217e-02 $\pi$ & 174 & 598.09 & 1684 & 2.26e-02 & 7.76e-02 & 2.54e-01  \\ 
200 & 4.6770e-02 $\pi$ & 295 & 690.63 & 1912 & 1.41e-01 & 3.49e-01 & 1.17e+00  \\ 
400 & 3.2684e-02 $\pi$ & 381 & 640.27 & 1382 & 7.08e-01 & 1.34e+00 & 4.14e+00  \\ 
500 & 2.9122e-02 $\pi$ & 405 & 556.26 & 1043 & 1.48e+00 & 2.40e+00 & 4.39e+00  \\ 
700 & 2.4519e-02 $\pi$ & 418 & 539.18 & 754 & 3.19e+00 & 4.34e+00 & 7.32e+00  \\ 
1000 & 2.0430e-02 $\pi$ & 433 & 505.31 & 610 & 6.52e+00 & 7.95e+00 & 1.57e+01  \\ 
\hline
\end{tabular}
\caption{Performance of Algorithm~\ref{alg:main} for computing a critical angle between two ellipsoidal cones with $10^3$ initial points.}\label{tab:elli_rspl}
\end{table}
\begin{table}[htb]
\centering
\begin{tabular}{|l|ccccccc|}
\hline
n    & $\theta_b(P,Q)$ & it$_{b}$ & it$_{a}$ & it$_{w}$ & CPU$_b$ & CPU$_a$ & CPU$_w$  \\ 
\hline
5 & 5.6950e-01 $\pi$ & 9 & 12.86 & 37 & 2.09e-02 & 2.53e-02 & 4.61e-02   \\ 
20 & 1.5928e-01 $\pi$ & 37 & 57.04 & 111 & 4.63e-02 & 7.57e-02 & 4.25e-01   \\ 
50 & 9.7918e-02 $\pi$ & 36 & 50.20 & 93 & 7.61e-02 & 1.10e-01 & 2.19e-01   \\ 
100 & 6.7218e-02 $\pi$ & 56 & 98.32 & 190 & 2.79e-01 & 5.44e-01 & 1.60e+00   \\ 
200 & 4.6773e-02 $\pi$ & 56 & 96.89 & 239 & 1.59e+00 & 2.86e+00 & 8.51e+00   \\ 
400 & 3.2693e-02 $\pi$ & 75 & 178.35 & 468 & 1.36e+01 & 3.32e+01 & 8.73e+01   \\ 
500 & 2.9136e-02 $\pi$ & 102 & 291.72 & 1051 & 3.52e+01 & 1.00e+02 & 3.54e+02   \\ 
700 & 2.4860e-02 $\pi$ & 99 & 255.17 & 973 & 9.23e+01 & 2.50e+02 & 1.00e+03   \\ 
1000 & 1.0850e-01 $\pi$ & 110 & 314.44 & 878 & 3.20e+02 & 9.79e+02 & 5.26e+03   \\ 
\hline
\end{tabular}
\caption{Performance of \ipopt~for computing a critical angle between two polyhedral cones with $10^3$ initial points.}\label{tab:elli_ip}
\end{table}

\begin{figure}[htp]
    \centering
    \includegraphics[width=0.7\textwidth, trim={5.3cm 9.5cm 5.3cm 10.5cm}]{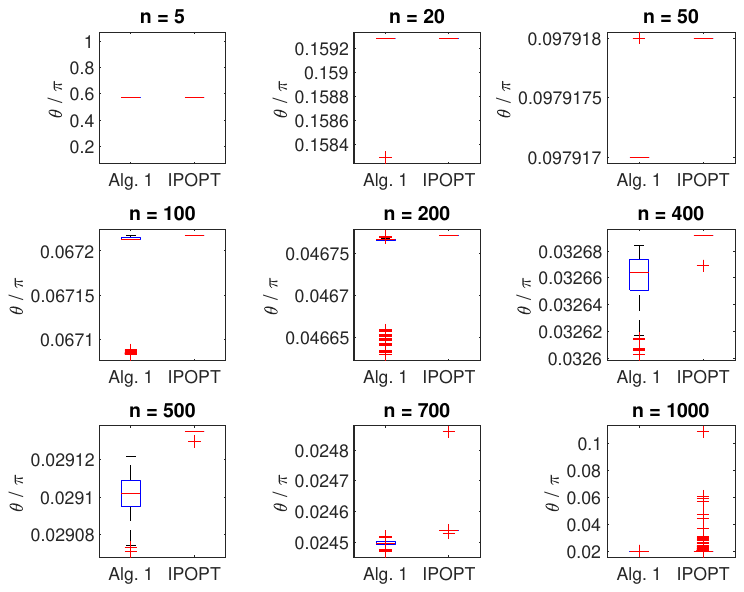}
    \caption{Box-plots of the critical angles between two ellipsoidal cones computed by Algorithm\,\ref{alg:main} and \ipopt~with $10^3$ initial points.}
    \label{fig:elli}
\end{figure}

\subsection{Critical angles between the SDP cone and the cone of nonnegative matrices}
Recall that $\mathcal S^n$ is the space of symmetric matrices of order $n$ equipped with the trace inner product $\langle A,B\rangle={\rm Tr}(AB)$. Furthermore,
$\mathcal P_n$ is the SDP cone, and $\mathcal N_n$ denotes the cone of nonnegative matrices in $\mathcal S^n$. That is,
\begin{align*}
    &\mathcal P_n=\{A\in\mathcal S^n\,:\,x^\top Ax\geq 0,\forall x\in\mathbb R^n\},\\
    &\mathcal N_n=\{B\in\mathcal S^n\,:\,\mbox{$B$ is nonnegative entrywise}\}.
\end{align*}
Let $\Theta(\mathcal P_n,\mathcal N_n)$ denote the maximal angle between $\mathcal P_n$ and $\mathcal N_n$.   The computation of $\Theta(\mathcal P_n,\mathcal N_n)$  was exhaustively treated in \cite{GSM} and \cite{SeSo2}. It is known that 
$$\Theta(\mathcal P_n,\mathcal N_n)=(3/4)\pi,\mbox{ for }n=2,3,4,\quad\mbox{and}\quad \lim_{n\to\infty}\Theta(\mathcal P_n,\mathcal N_n)=\pi.$$
The exact computation of $\Theta(\mathcal P_n,\mathcal N_n)$ for $n\geq 5$ becomes difficult, and only lower bounds for certain values of $n$ are known. As it is discussed in \cite{HS}, a motivation to compute $\Theta(\mathcal P_n,\mathcal N_n)$ is that it provides  a lower bound for the greatest possible angle between two matrices in the copositive cone $\mathcal C_n:=\{A\in\mathcal S^n\,:\,x^\top Ax\geq 0,\forall x\in\mathbb R^n_+\}$, cf. \cite{GSM}. Our goal in this section is to use Algorithm\,\ref{alg:main} to improve the knowledge of these lower bounds.

It is clear that $\mathcal P_n$ is a LISC cone since it is the symmetric cone of the algebra $\mathcal S^n$. In this case, its representation is 
$$\mathcal P_n=G(\mathcal P_n)$$
with $G:\mathcal S^n\to\mathcal S^n$ the identity map.

The cone $\mathcal N_n$ is also LISC. Indeed, it can be expressed as 
$$\mathcal N_n=H(\mathbb R^N_+),$$
where $N=n(n+1)/2$ and $H:\mathbb R^N\to\mathcal S^n$ is the linear map constructed as follows: Let $\mathcal E:=\{e_1,\ldots,e_N\}$ be the canonical basis of $\mathbb R^N$ and let $\{E_{i,j}\}_{i\leq i\leq j}$ be the canonical basis of $\mathcal S^n$. That is, $E_{i,j}\in\mathcal S^n$ is the matrix whose entries are zeros except the entries corresponding to $(i,j)$ and $(j,i)$ in which take values equal to one. It is clear that $H$ is completely defined by knowing the image of each element of its basis $\mathcal E$. In this way, we take the following assignment: for $k=1,\ldots,N$,
$$H(e_k)=E_{i_k,j_k},$$
where $i_k:=2k-\ell_k(\ell_k-1)/2$ and $j_k:=\mathop{\rm argmin}\limits_{1\leq\ell\leq n}\{\ell(\ell+1)/2:\ell(\ell+1)/2\geq k\}.$ Hence, for $y\in\mathbb R^N$, the value of $H(y)$ is nothing but
$$H(y)=\sum_{k=1}^N y_k E_{i_k,j_k}.$$
The mapping $H$ is invertible since it maps the canonical basis of $\mathbb R^N$ onto the canonical basis of $\mathcal S^n$. Hence, assumptions \ref{assump1} and \ref{assump2} are satisfied and $\mathcal N_n$ is a LISC cone.

In the implementation of Algorithm\,\ref{alg:main}, we also need to know the map $H^T:\mathcal S^n\to\mathbb R^N$ (the adjoint map of $H$). It is not difficult to see that it is given by
$$H^T(X)=(\langle X,E_{i_1,j_1}\rangle ,\ldots, \langle X,E_{i_N,j_N}\rangle)^T.$$

In practice, the numerical evaluation of $H(y)$ and $H^T(X)$  is easy. To illustrate it, consider $n=3$. Then $N=6$ and we have:
$$H\begin{pmatrix}
    y_1\\y_2\\y_3\\y_4\\y_5\\y_6
\end{pmatrix}=\begin{pmatrix}
    y_1&y_2&y_4\\y_2&y_3&y_5\\y_4&y_5&y_6
\end{pmatrix}\quad\mbox{and}\quad H^T\begin{pmatrix}
    x_{11}&x_{12}&x_{13}\\ x_{12}&x_{22}&x_{23}\\x_{13}&x_{23}&x_{33}
\end{pmatrix}=\begin{pmatrix}
    x_{11}\\2x_{12}\\x_{22}\\2x_{13}\\2x_{23}\\x_{33}
\end{pmatrix}.$$

Now, we are ready to implement Algorithm\,\ref{alg:main} to estimate $\Theta(\mathcal P_n,\mathcal N_n)$.  Algorithm\,\ref{alg:main} generates the points $(X^k,y^k)$. Recall that $X^k$ is a matrix in $\mathcal S^n$ and $y^k$ is a vector in $\mathbb R^N$. In Step $2$ of Algorithm~\ref{alg:main}, for $\mu_1>0$, problem \eqref{lin1} becomes
 \begin{equation*}\label{step2SDP}
\begin{array}{ll}
\min&\displaystyle \frac{1}{2}\Vert X-B^k\Vert^2\\
\mbox{s.t.}&X\in\mathcal P_n,\\
&{\rm Tr}(X)=1,
\end{array}
\end{equation*} 
where $B^k:=X^k-(1/\mu_1)(Hy^k-\delta_k\Vert X^k\Vert^{-1}\Vert Hy^k\Vert X^k)$. If $B^k$ has the spectral decomposition, the solution to the above problem is $X^k=Q{\rm Diag}(\lambda(B^k)Q^T$, with $Q$ being some orthogonal matrix of order $n$. From Example\,\ref{ex:pyeccion} we have that  $X^k=Q{\rm Diag}(\bar x)Q^T$, where $\bar x\in\mathbb R^n$ is the solution of
\begin{equation*}
\begin{array}{ll}
\min&\displaystyle \frac{1}{2}\Vert x-\lambda(B^k)\Vert^2\\
\mbox{s.t.}&x\in \mathbb R^n_+,\\
&\langle {\bf 1}_n,x\rangle=1.
\end{array}
\end{equation*} 
When $\mu_2>0$, problem \eqref{lin2} becomes
\begin{equation*}
\begin{array}{ll}
\min&\displaystyle \frac{1}{2}\Vert y-b^k\Vert^2\\
\mbox{s.t.}&y\in \mathbb R^N_+,\\
&\langle {\bf 1}_N,y\rangle=1,
\end{array}
\end{equation*}
where $b^k:=y^k-(1/\mu_2)H^T(X^k-\delta_k\Vert  X^k\Vert\Vert Hy^k\Vert^{-1} Hy^k)$.

The tolerances for the algorithm stopping test are $\epsilon_1=10^{-6}$, $\epsilon_2=10^{-6}$ and $\epsilon_3=10^{-7}$.
We consider a maximum number of iterations equal to $5000$, and the value of the prox-parameters are chosen as $\mu_1 = 0.01$ and $\mu_2 = 5$.  For the initial point $(X^0,y^0)$, we set as $X^0=Diag(x^0)$ and, we generated $x^0$ and $y^0$ as two random points in the simplex. The largest dimension for the instances is $60$, because the subproblems to be solved at each iteration are more complex. In particular, the first requires the spectral decomposition of a matrix and  a subsequent projection on the $n$-dimensional standard simplex. The second subproblem also consists in projecting on the standard simplex, but this time the dimension is much higher, that is $N = n(n+1)/2$ (e.g., with $n = 60$, $N=1830$).

In Table~\ref{tab:PN}, we report the best critical angle $\theta_b(\mathcal P_n,\mathcal N_n)$ obtained by using Algorithm~\ref{alg:main} for different values of $n$. In the second column of this table, titled `$\theta_{lb}(\mathcal P_n,\mathcal N_n)$', we additionally show the values computed by using the procedure in \cite{SeSo2} which only guarantees a lower bound for the maximal angle between the two considered cones. We solve the optimization problem given in \cite{SeSo2} by \ipopt~and considering $10^{3}$ random initial points. Then, in the table, we provide the best value computed.
We can see that our algorithm could find a slightly better solution for $n=50$ requiring a very small computational time.  However, the reported values improve significantly what was known in \cite[Table\,2]{SeSo2}.   We do not compare our algorithm with \ipopt~applied to FP~\eqref{maxfrac}~because this solver is not applicable in this more complex setting: the Cones $\mathcal{P}_n$ and $\mathcal{N}_n$ do not have a favorable structure for off-the-shelf NLP solvers.

\begin{table}[]
\begin{tabular}{|l|c|c|c|c|c|c|c|c|}
\hline
n    & \multicolumn{1}{l|}{$\theta_{lb}(\mathcal P_n,\mathcal N_n)$} & $\theta_b(\mathcal P_n,\mathcal N_n)$ & it$_{b}$ & it$_{a}$ & it$_{w}$ & CPU$_b$  & CPU$_a$  & CPU$_w$  \\ \hline
4           & 0.7500  $\pi$      &  0.7500 $\pi$& 18 & 35.96 & 322 & 7.29e-04 & 1.76e-03 & 3.54e-02  \\ 
5           & 0.7575  $\pi$      &  0.7575 $\pi$& 20 & 102.32 & 536 & 8.87e-04 & 5.84e-03 & 3.23e-02  \\ 
6           & 0.7575  $\pi$      &  0.7575 $\pi$& 19 & 125.89 & 479 & 1.10e-03 & 8.35e-03 & 2.78e-02  \\ 
7           & 0.7575  $\pi$      &  0.7575 $\pi$& 22 & 202.35 & 478 & 1.16e-03 & 1.27e-02 & 3.34e-02  \\ 
8           & 0.7608  $\pi$      &  0.7608 $\pi$& 23 & 237.76 & 635 & 1.26e-03 & 1.59e-02 & 4.22e-02  \\ 
9           & 0.7608  $\pi$      &  0.7608 $\pi$& 20 & 243.34 & 530 & 1.27e-03 & 1.61e-02 & 3.95e-02  \\ 
10          & 0.7609  $\pi$      &  0.7609 $\pi$& 25 & 256.06 & 750 & 1.57e-03 & 1.73e-02 & 4.47e-02  \\ 
11          & 0.7627  $\pi$      &  0.7627 $\pi$& 25 & 260.09 & 870 & 1.97e-03 & 1.80e-02 & 5.89e-02  \\ 
12          & 0.7649  $\pi$      &  0.7649 $\pi$& 141 & 273.30 & 903 & 9.53e-03 & 1.90e-02 & 6.14e-02  \\ 
13          & 0.7649  $\pi$      &  0.7649 $\pi$& 130 & 282.42 & 1114 & 9.26e-03 & 2.09e-02 & 7.03e-02  \\ 
14          & 0.7659  $\pi$      &  0.7659 $\pi$& 123 & 304.45 & 927 & 8.73e-03 & 2.27e-02 & 7.51e-02  \\ 
15          & 0.7678  $\pi$      &  0.7678 $\pi$& 137 & 327.89 & 1304 & 1.04e-02 & 2.54e-02 & 9.36e-02  \\ 
16          & 0.7699  $\pi$      &  0.7699 $\pi$& 133 & 345.00 & 1835 & 1.06e-02 & 2.79e-02 & 1.40e-01  \\ 
17          & 0.7699  $\pi$      &  0.7699 $\pi$& 128 & 350.01 & 1958 & 1.17e-02 & 3.09e-02 & 1.60e-01  \\ 
18          & 0.7699  $\pi$      &  0.7699 $\pi$& 122 & 350.11 & 2179 & 1.23e-02 & 3.71e-02 & 1.96e-01  \\ 
19          & 0.7703  $\pi$      &  0.7703 $\pi$& 129 & 355.40 & 1548 & 1.24e-02 & 3.69e-02 & 1.72e-01  \\ 
20          & 0.7719  $\pi$      &  0.7719 $\pi$& 135 & 390.75 & 1741 & 1.37e-02 & 3.96e-02 & 2.91e-01  \\ 
21          & 0.7719  $\pi$      &  0.7719 $\pi$& 130 & 396.87 & 1955 & 1.52e-02 & 4.27e-02 & 1.91e-01  \\ 
22          & 0.7719  $\pi$      &  0.7719 $\pi$& 136 & 420.58 & 1724 & 1.57e-02 & 4.77e-02 & 1.83e-01  \\ 
23          & 0.7722  $\pi$      &  0.7722 $\pi$& 142 & 412.98 & 1593 & 1.73e-02 & 4.99e-02 & 1.76e-01  \\ 
24          & 0.7735  $\pi$      &  0.7735 $\pi$& 150 & 420.40 & 1831 & 1.81e-02 & 5.10e-02 & 2.03e-01  \\ 
25          & 0.7735  $\pi$      &  0.7735 $\pi$& 147 & 425.24 & 1480 & 2.00e-02 & 6.42e-02 & 2.59e-01  \\ 
26          & 0.7735  $\pi$      &  0.7735 $\pi$& 133 & 458.19 & 1454 & 3.25e-02 & 1.29e-01 & 4.12e-01  \\ 
27          & 0.7739  $\pi$      &  0.7739 $\pi$& 156 & 456.77 & 1997 & 3.93e-02 & 1.41e-01 & 5.50e-01  \\ 
28          & 0.7750  $\pi$      &  0.7750 $\pi$& 158 & 461.31 & 1764 & 4.14e-02 & 1.38e-01 & 5.63e-01  \\ 
29          & 0.7750  $\pi$      &  0.7750 $\pi$& 149 & 492.23 & 1704 & 4.21e-02 & 1.52e-01 & 5.08e-01  \\ 
30          & 0.7757  $\pi$      &  0.7757 $\pi$& 148 & 504.14 & 1834 & 4.04e-02 & 1.60e-01 & 6.21e-01  \\ 
40 & 0.7789 $\pi$ & 0.7789 $\pi$ & 275 & 787.20 & 2197 & 1.27e-01 & 4.29e-01 & 1.32e+00  \\ 
50          &     0.7809 $\pi$             &  0.7812 $\pi$& 356 & 1193.38 & 3108 & 2.50e-01 & 9.71e-01 & 2.65e+00  \\ 
60          & 0.7837  $\pi$                & 0.7837 $\pi$& 507 & 1553.64 & 4189 & 5.57e-01 & 1.78e+00 & 6.93e+00 \\
\hline
\end{tabular}
 \caption{Performance of Algorithm~\ref{alg:main} for computing a critical angle between the SDP cone and the cone of nonegative matrices.}\label{tab:PN}
\end{table}

\section{Conclusions}\label{sec:conc}

In this work, we have investigated the problem of computing critical angles between pairs of convex cones in finite-dimensional Euclidean spaces. Our contributions are summarized as follows:
a) we have shown that the notion of critical angles considered in the literature is equivalent to that of (KKT) stationarity for a class of nonlinear optimization problems;
b) we have posed the problem in the fractional programming (FP) setting and proposed a specialized algorithm to compute stationary points, yielding thus critical angles to our original problem;
c) we have further specialized our analysis to the broad class of (LISC) convex cones derived as linear images of symmetric cones. Moreover, we have demonstrated that our FP formulation is well-defined and suitable for numerical optimization.

Our algorithm is the first numerical procedure for computing critical angles between pairs of convex cones. The approach requires solving two independent subproblems at every iteration to compute (with the aid of the Armijo line search) a new iterate. For the class of LISC cones, we have shown that solutions to such subproblems are projections onto well-structured convex sets. Otherwise, for more general cones, we can still apply off-the-shelf (linear, quadratic, conic, etc.) convex solvers to solve such subproblems. Our theoretical analysis demonstrates that our approach asymptotically computes critical angles regardless of initialization.

Numerical experiments on different types of convex cones (polyhedral, ellipsoidal, and Loewnerian cones) in several dimensions illustrate the practical performance of our algorithm. Comparison with \ipopt~have shown that exploiting the problem's structure as done by our approach pays off: for the considered test problems, our algorithm is up to one hundred times faster than \ipopt~while maintaining (sometimes even improving) the quality of computed critical angles.

\section*{Data availability}
All data generated or analysed during this study are included in this article.

\end{document}